\documentclass[11pt,twoside,a4paper]{article}
\usepackage{amsmath}
\usepackage{amssymb}
\usepackage{amsthm}
\usepackage{enumerate}
\usepackage{float}
\usepackage{geometry}
\usepackage{graphicx}
\usepackage{indentfirst}
\usepackage{listings}
\usepackage{mathtools}
\usepackage{setspace}
\usepackage{url}
\usepackage{threeparttable}
\usepackage{multirow}
\usepackage{longtable}
\usepackage{algorithm}
\usepackage{algorithmic}
\usepackage{siunitx}
\usepackage{enumitem}
\usepackage{letltxmacro}
\usepackage{textcase}
\usepackage{mfirstuc}
\usepackage{arydshln}
\usepackage[numbers,compress]{natbib}
\theoremstyle{definition}
\newtheorem{theorem}{Theorem}[subsection]

\newtheorem{definition}[theorem]{Definition}
\newtheorem{proposition}[theorem]{Proposition}

\newtheorem{corollary}[theorem]{Corollary}

\theoremstyle{remark}

\author{Shen Zhang\\ {shen-zha20@mails.tsinghua.edu.cn}}
\title{Stability of Persistent Path Diagrams}
\date{May 18, 2024}

\begin{document}
\setstretch{1.3}
\linespread{1.0}
\maketitle

\begin{abstract}
    In real-world systems, the relationships and connections between components are highly complex. Real systems are often described as networks, where nodes represent objects in the system and edges represent relationships or connections between nodes. With the deepening of research, networks have been endowed with richer structures, such as directed edges, edge weights, and even hyperedges involving multiple nodes.
  
    Persistent homology is an algebraic method for analyzing data. It helps us understand the intrinsic structure and patterns of data by tracking the death and birth of topological features at different scale parameters.The original persistent homology is not suitable for directed networks. However, the introduction of path homology established on digraphs solves this problem. This paper studies complex networks represented as weighted digraphs or edge-weighted path complexes and their persistent path homology. We use the homotopy theory of digraphs and path complexes, along with the interleaving property of persistent modules and bottleneck distance, to prove the stability of persistent path diagram with respect to weighted digraphs or edge-weighted path complexes. Therefore, persistent path homology has practical application value.
\end{abstract}

\section{Introduction}

\subsection{Networks and Their Extensions}
In real-world systems, the relationships and connections among components are often very complex. Networks are important tools for describing complex systems. A network is a graph composed of nodes and edges connecting two nodes, where nodes represent certain objects in the system, and edges represent the relationships or connections between nodes. Initially, networks were simple unweighted undirected graphs, not considering directed or weighted relationships between components. As research on real-world systems deepened, scientists began to endow networks with more complex structures, such as using directions to describe the start and end points of edges, and using weights to describe the strength of edges.

These extensions help better represent complex systems, but \citet{boccaletti2023structure} pointed out that graphs themselves are not sufficient to provide a complete description of complex systems: in graphs, edges only allow connections between two nodes, and there is no structure connecting multiple nodes. Therefore, graphs can only capture pairwise relationships and connections. However, group interactions widely exist in complex systems in the real world. For example, ecological studies show that competition for food and territory often occurs among groups of three or more species \cite{levine2017beyond}; neural dynamics show that behaviors that need to be predicted require interactions among multiple neurons \cite{ganmor2011sparse}; social mechanisms, such as peer pressure or cooperation, also go beyond the concept of binary connections \cite{benson2018simplicial}.

In fact, there has been research involving higher-order relationships among multiple nodes in networks. For example, the majority rule model in opinion dynamics \cite{montes2011majority}, the public goods game in evolutionary game theory \cite{szolnoki2010reward}, and complex contagions \cite{centola2010spread} all consider interactions among multiple objects. However, these models attempt to decompose higher-order relationships into binary relationships, using the language of graph structures to describe higher-order structures, essentially not going beyond graph structures \cite{battiston2020networks}. Considering that describing higher-order structures in systems can enhance our modeling capabilities and help understand and predict their dynamic behaviors, we need to explore methods to define, model, and construct higher-order networks.

\subsection{Topological Data Analysis}

Considering that nodes and edges are 0-simplices and 1-simplices respectively, if we want to extend the graph structure to higher dimensions to describe higher-dimensional structures and relationships, simplicial complexes are a natural choice. Simplicial complexes have rich topological and algebraic structures, and an important achievement in data analysis using algebraic topology is persistent homology proposed by \citet{de2004topological}. The elementary  idea is to consider a family of parameterized simplicial complexes indexed by a scale parameter, and track the creation and disappearance of topological features as the scale parameter changes. A simplicial complex can be roughly considered as a collection of points, edges, triangles, tetrahedra, and higher-dimensional polyhedra, and persistent homology can be roughly considered as tracking the changes in the number of connected components, holes, and voids in simplicial complexes at different scales. Persistent homology can be represented by persistent barcodes or persistence diagrams, where each barcode or point corresponds to a homology class appearing during the change of the scale parameter and represents the scale range where the homology class exists.

For example, consider a finite set of points \(V\) in \(\mathbb{R}^2\), where \(\mathbb{R}^2\) is equipped with the \(L^2\) norm inducing the distance \(d\).

For different distance parameters \(\delta \in \mathbb{R}\), we construct
\begin{align}
    \mathfrak{R}_{V}^{\delta}:=\left\{ \sigma \subseteq V \mid \sigma \neq \varnothing, \max _{x, x^{\prime} \in \sigma} d\left(x, x^{\prime}\right) \leq \delta\right\}
\end{align}
By definition, for \(\delta \le \delta^\prime\), we have \(\mathfrak{R}_{V}^{\delta} \subset \mathfrak{R}_{V}^{\delta^\prime}\). Thus, we obtain a nested sequence of simplicial complexes, which we call the Vietoris–Rips filtration \cite{vietoris1927hoheren}.

We also construct
\begin{align}
    \mathfrak{C}_{V}^{\delta}:=\left\{ \sigma \subseteq V \mid \sigma \neq \varnothing, \cap_{x \in \sigma} B\left(x,\frac{\delta}{2}\right) \neq \emptyset \right\}
\end{align}
By definition, for \(\delta \le \delta^\prime\), we have \(\mathfrak{C}_{V}^{\delta} \subset \mathfrak{C}_{V}^{\delta^\prime}\). Thus, we obtain a nested sequence of simplicial complexes, which we call the Čech filtration \cite{vcech1959topologicke}.

For a positive integer \(p \in \mathbb{Z}_{>0}\), we consider the \(p\)-th simplicial homology group of each simplicial complex in the filtration, thus obtaining a sequence of homology groups. These homology groups can be regarded as modules over the ring of integers, and we call this sequence of modules the persistent module \cite{edelsbrunner2002topological}. Moreover, we can also consider the case where coefficients are taken in a general ring, and by appropriate selection of the ring, we can equivalently represent the persistent module as a persistence diagram \cite{edelsbrunner2008persistent}. It can be proved that the persistence diagrams of the Vietoris–Rips filtration and the Čech filtration are stable with respect to perturbations of edge weights \cite{chazal2014persistence}, making the application of persistent homology in data analysis practically meaningful.

Based on algebraic topology, the method of persistent homology provides a clear theoretical framework for studying data with complex structures and has been widely applied in fields such as image signals, disease spread, materials science, finance, neuroscience, biochemistry, and natural language \cite{otter2017roadmap}.

\subsection{GLMY Path Homology Theory}
The current issue is that traditional persistent homology theory is based on simplicial complexes and simplicial homology, lacking directionality. Can we introduce higher-dimensional structures while preserving directionality? Furthermore, our previous construction of filtrations assumed that data lie in a metric space, which is not always the case or natural. Therefore, let us consider the general case of a finite set of points. In the case of a weighted directed graph, a network refers to a finite set \(X\) and a weight function \(w: X \times X \rightarrow [0,+\infty]\), such that for any \(x, y \in X\), we have \(w(x, y) \geq 0\) and \(w(x, y) = 0 \Longleftrightarrow x = y\). In the case of a weighted undirected graph, we do not distinguish between \(w(x, y)\) and \(w(y, x)\), so we also require \(w(x, y) = w(y, x)\), in which case we also refer to the network as a symmetric network. The graph structure and edge weights in a network can be regarded as analogs of metric spaces and their metrics, but the edge weights in directed graphs do not require the symmetry assumption required by metrics.

Replacing the metric with edge weights, we can extend the filtrations previously considered in metric spaces to a general finite set of points \(V\).

\begin{enumerate}
    \item Consider the Vietoris–Rips filtration
          \begin{align}
              \mathfrak{R}_{V}^{\delta}:=\left\{ \sigma \subseteq V \mid \sigma \neq \varnothing, \max _{x, x^{\prime} \in \sigma} w\left(x, x^{\prime}\right) \leq \delta\right\}
          \end{align}
          Note that the Vietoris–Rips filtration implicitly symmetrizes the weights. More specifically, given a network \(\left(X, w\right)\), if we assign \(\max \left(w\left(x, x^{\prime}\right), w\left(x^{\prime}, x\right)\right)\) to \(w\left(x, x^{\prime}\right), w\left(x^{\prime}, x\right)\) for all \(x, x^{\prime} \in X\), the resulting Vietoris–Rips filtration remains unchanged because a simplex \(\left[x, x^{\prime}\right]\) appears at \(\max \left(w\left(x, x^{\prime}\right), w\left(x^{\prime}, x\right)\right)\), and so on.
    \item Consider the Dowker filtration
          \begin{align}
              \mathfrak{D}^\delta_{V}:=\left\{ \sigma \mid \exists x^\prime \in V \text{ s.t. } w(x_i, x^\prime) \leq \delta \right\}
          \end{align}
          or equivalently
          \begin{align}
              \mathfrak{D}^\delta_{V}:=\left\{ \sigma \mid \exists x^\prime \in V \text{ s.t. } w(x^\prime, x_i) \leq \delta \right\}
          \end{align}
          or equivalently
          \begin{align}
              \mathfrak{D}^\delta_{V}:=\left\{ \sigma \mid \exists x^\prime, x^{\prime \prime} \in V \text{ s.t. } w(x^\prime, x_i) \leq \delta, w(x_i, x^{\prime \prime}) \leq \delta \right\}
          \end{align}
          It can be proved that in both cases we eventually obtain the same persistent diagram, and in the case of symmetric networks, the Dowker filtration and the Čech filtration are consistent\cite{chowdhury2016persistent}.
\end{enumerate}

In each case, we obtain a simplicial complex filtration, to which we can apply the simplicial homology functor and obtain persistent modules, also known as the Vietoris–Rips/Dowker persistent homology method. As \cite{chowdhury2016persistent} ponited out, the Vietoris–Rips method implicitly symmetrizes the weights; on the other hand, the Dowker method also has implicit symmetrization, since both Vietoris–Rips and Dowker methods apply simplicial homology, where simplices with different vertex orders are mapped to the same one-dimensional linear subspace in the chain group. Thus, the Vietoris–Rips/Dowker persistent homology method is insensitive to asymmetry to varying degrees.

Another method for handling the homology of directed graphs is to use Hochschild homology.\cite{hochschild1945cohomology} In fact, paths on directed graphs have a natural product operation, which allows us to define the path algebra of directed graphs, making Hochschild homology a natural and important object.\cite{grigor2012homologies} However, Hochschild homology is trivial for orders greater than 2\cite{happel2006hochschild}, so the idea of constructing persistent homology using this method is not attractive.

To construct persistent homology on weighted directed graphs, this paper uses the GLMY path homology established in a series of papers published between 2012 and 2020 by A. Grigor'yan, Y. Lin, Y. Muranov, and S.T. Yau.\cite{grigor2012homologies,grigor2015cohomology,grigor2019homology,grigor2024analytic,grigor2014graphs,grigor2020path,grigor2014homotopy}

Compared to previously studied concepts of graph homology, the path homology of directed graphs introduced in this paper has many advantages\cite{grigor2012homologies}:
\begin{enumerate}
    \item Path homology in all dimensions may be non-trivial.
    \item The chain complex associated with the path complex has a richer structure than simplicial chain complexes, such as binary hypercubes and many other subgraphs.
    \item Functoriality, for example, directed graph mappings induce group homomorphisms between path homology groups.
\end{enumerate}

\subsection{Structure Arrangement}
In this paper, we first consider networks as directed graphs with edge weights and no loops, also referred to as weighted directed graphs. Given a weighted directed graph, the process of data analysis using persistent path homology consists of the following parts \cite{chaplin2022grounded}:

\begin{enumerate}
    \item Use a filtration map $F: \operatorname{Obj}\left(\mathbf{WDgr}\right) \rightarrow \operatorname{Obj}\left(\left[\mathbf{R}, \mathbf{Dgr}\right]\right)$ to generate a family of directed graphs indexed by scale parameters from the weighted directed graph.
    \item Use the functor $P$ to generate edge-weighted path complexes from the weighted directed graph.
    \item Use the chain complex functor $\Omega_*: \mathbf{PC} \rightarrow \mathbf{Ch}$ to map path complexes to path complex chain complexes and induce chain mappings from weak mappings of path complexes.
    \item Select the $k$-th homology and use the functor $H_{k}: \mathbf{Ch} \rightarrow \mathbf{Vec} $ to generate the $k$-th path homology group from the path complex chain complex.
    \item Calculate the persistence diagram $D:\operatorname{Obj}(\mathbf{Persvec})\to\mathbf{Mult}\left(\overline{\mathbb{R}}^2\right) $ to obtain the multi-point set in the extended $\mathbb{R}$ plane.
\end{enumerate}
\begin{align}
    \mathbf{WDgr}\xrightarrow{F}\left[\mathbf{R}, \mathbf{Dgr}\right]\xrightarrow{\left[\mathbf{R}, P\right]}\left[\mathbf{R}, \mathbf{PC}\right] \xrightarrow{\left[\mathbf{R}, \Omega_*\right]}\left[\mathbf{R}, \mathbf{Ch}\right] \xrightarrow{\left[\mathbf{R}, H_{k}\right]}\left[\mathbf{R}, \mathbf{Vec}\right]\xrightarrow{D} \mathbf{Mult}\left(\overline{\mathbb{R}}^2\right)
\end{align}

In this case, we assign higher-order structures to the directed graph by generating path complexes, produce scale parameters from weights via filtration mappings, and ultimately compute the persistence diagram as a generalized quantity describing the topological features varying with the scale parameters.

In practical applications, we often face situations where part of the network itself has complex higher-order structures and can be viewed as path complexes; additionally, sometimes we need to adjust and filter the higher-order structures of path complexes generated by directed graphs. In these cases, traditional methods may not be directly applicable. Therefore, more generally, we consider path complexes with edge weights. In this case, we can omit the step of generating path complexes from directed graphs and instead consider path complexes directly, yielding:

\begin{align}
    \mathbf{WPC}\xrightarrow{F}\left[\mathbf{R}, \mathbf{PC}\right] \xrightarrow{\left[\mathbf{R}, \Omega_*\right]}\left[\mathbf{R}, \mathbf{Ch}\right] \xrightarrow{\left[\mathbf{R}, H_{k}\right]}\left[\mathbf{R}, \mathbf{Vec}\right]\xrightarrow{D} \mathbf{Mult}\left(\overline{\mathbb{R}}^2\right)
\end{align}

To ensure the value of persistent path homology in practical applications, we naturally hope that the persistence diagram as a general quantity remains stable against noise and perturbations in the input data. Under the setting of weighted complete digraph, the stability of the persistent diagram of path persistence homology with respect to weights has been demonstrated by \cite{chowdhury2018persistent}. However, this condition is often not met in practice. This paper further considers edge-weighted path complexes based on weighted directed complete graphs and provides stability results and proofs for the persistent diagram of path persistence homology with respect to the weights and structure of complex networks. The conditions are relaxed to include directed graph homotopy equivalence or weak homotopy equivalence of path complexes.

We will first introduce the GLMY theory related to path homology\cite{grigor2012homologies} and path homotopy\cite{grigor2019homology, grigor2014homotopy}, followed by the introduction of persistent homology\cite{chazal2016structure} and its combination with path homology\cite{chowdhury2018persistent,chaplin2022grounded}, and finally introduce the metrics of persistence diagrams\cite{chazal2009proximity,bauer2014induced} and prove the stability of persistent path homology. Readers who are already familiar with the fundamental concepts can skip the corresponding sections.

\section{Path Complexes and Path Homology}

\subsection{Path Spaces}
Consider any non-empty finite set $V$, where its elements are referred to as vertices. Let $\mathbb{K}$ be any fixed field.
\begin{definition}[Elementary  Path]
    For a non-negative integer $n$, we call a sequence $\{i_k\}^n_{k=0}$ composed of vertices from $V$ a elementary  $n$-path on the set $V$. There is no requirement for the vertices in the path to be distinct. For $n=-1$, a elementary  $n$-path is an empty set $\emptyset$.
    We also denote a elementary  $n$-path $\{i_k\}^n_{k=0}$ as $i_0\ldots i_n$.
\end{definition}
\begin{definition}[Path]
    Let $\Lambda_n=\Lambda_n(V,\mathbb{K})$ be the $\mathbb{K}$-linear space containing all $\mathbb{K}$-linear combinations of elementary  $n$-paths. Elements in $\Lambda_n$ are called $n$-paths on $V$.
\end{definition}

For a non-negative integer $n$, when a elementary  $n$-path $i_0\ldots i_n$ is considered as an element in $\Lambda_n$, we denote it as $e_{i_0\ldots i_n}$. When the empty set $\emptyset$ is considered as an element in $\Lambda_{-1}$, we denote it as $e$.

It is noteworthy that $\Lambda_{-1}\cong\mathbb{K}$, and we set $\Lambda_{-2}=\{0\}$.
\subsection{Boundary Operator}
\begin{definition}[Non-Regular Boundary Operator]
    For $n\geq 0$, the non-regular boundary operator $\partial^{\text{nr}}:\Lambda_{n}\to\Lambda_{n-1}$ is a linear operator defined on elementary  paths as follows:
    \begin{align}
        \partial^{\text{nr}} e_{i_0\ldots i_n}=\sum_{q=0}^n(-1)^q e_{i_0\ldots\hat{i_q}\ldots i_n}
    \end{align}
    where $\hat{i_q}$ denotes the omission of the index $i_q$. We set $\partial:\Lambda_{-1}\to\Lambda_{-2}$ as $0$.
\end{definition}

Thus, we obtain the path chain complex $\Lambda_*$ as follows:
\begin{align}
    \ldots\rightarrow \Lambda_{n}\rightarrow \Lambda_{n-1}\rightarrow\ldots\rightarrow\Lambda_{0}\rightarrow\mathbb{K}\rightarrow 0
\end{align}
where all arrows are given by the non-regular boundary operator $\partial^{\text{nr}}$.

\begin{theorem}[{\citet{grigor2012homologies}}]
    $\partial^{\text{nr}}\circ\partial^{\text{nr}}=0$
\end{theorem}

\subsection{Regular Paths}
\begin{definition}[Regular and Non-Regular Elementary  Paths]
    We call a elementary  path $i_1\ldots i_n$ non-regular if $i_{k-1}=i_k$ for some $k=1,\ldots,n$, otherwise, it is called regular.
\end{definition}

For $n\geq -1$, consider the subspace $\mathcal{R}_{n}$ spanned by regular $n$-paths over the field $\mathbb{K}$:
\begin{align}
    \mathcal{R}_{n}=\mathcal{R}_{n}(V):=\operatorname{span}_{\mathbb{K}}\{e_{i_{0}\ldots i_{n}}\mid i_{0}\ldots i_{n}\text{ regular}\}
\end{align}
and the subspace $I_{n}$ spanned by non-regular $n$-paths over the field $\mathbb{K}$:
\begin{align}
    I_{n}=I_{n}(V):=\operatorname{span}_{\mathbb{K}}\{e_{i_{0}\ldots i_{n}}\mid i_{0}\ldots i_{n}\text{ non-regular}\}
\end{align}
For $n=-2$, we set $\mathcal{R}_{-2}=I_{-2}=\{0\}$.

We have $\Lambda_n=\mathcal{R}_n\oplus I_n$, hence $\widetilde{\mathcal{R}_n}:=\Lambda_n/I_n\cong \mathcal{R}_n$.

\begin{definition}[Regular Paths]
    Elements in $\mathcal{R}_n$ are called regular paths.
\end{definition}

\begin{definition}[Regular Boundary Operator]
    We define the regular boundary operator on $\mathcal{R}_n$ by pulling back through the natural linear isomorphism $\mathcal{R}_n\to \widetilde{\mathcal{R}_n}$.
\end{definition}

We denote the boundary operator on $\mathcal{R}_{n}$ as the regular boundary operator, also denoted as $\partial$. Correspondingly, we denote the boundary operator on $\Lambda_{n}$ as the non-regular boundary operator, denoted as $\partial^{\text{nr}}$.

We have the following regular chain complex $\mathcal{R}_*$:
\begin{align}
    \ldots\rightarrow \mathcal{R}_{n}\rightarrow \mathcal{R}_{n-1}\rightarrow\ldots\rightarrow\mathcal{R}_{0}\rightarrow\mathbb{K}\rightarrow 0
\end{align}
where all arrows are given by the regular boundary operator $\partial$.

\subsection{Path Complexes and Their Generation}
Now we formally give the definition of path complexes.
\begin{definition}[Path Complex]
    For a set $V$, a path complex is a set $P$ of non-empty elementary  paths over $V$, such that for any $n \geq 0$, if $i_{0} \ldots i_{n} \in P$, then the truncated paths $i_{0} \ldots i_{n-1}$ and $i_{1} \ldots i_{n}$ are also in $P$.
\end{definition}

The set of $n$-elementary  paths in $P$ is denoted as $P_{n}$. Thus, the path complex $P$ can be viewed as the set $\left\{P_{n}\right\}_{n=-1}^{\infty}$ satisfying the above property. For a fixed path complex $P$, all elementary  paths coming from $P$ are called allowed, while all elementary  paths not in $P$ are called disallowed.

In particular, the set $P_{-1}$ contains an empty elementary  path $e$. Elements in $P_{0}$ (i.e., allowed 0-paths) are called vertices of $P$. $P_{0}$ is a subset of $V$. By definition, if $i_{0} \ldots i_{n} \in P$, then all $i_{k}$ are vertices. Thus, we can (and will) remove all non-vertices from the set $V$, so that $V=P_{0}$. Elements in $P_{1}$ are called edges of $P$.

We naturally consider the path complex generated by a set of paths.
\begin{definition}[Path Complex Generated by a Set of Paths]
    For a vertex set $V$, consider a finite subset $A$ of $\cup_{n=0}^\infty \Lambda_n\left(V\right)$, define
    \begin{align}
        \Delta A=\left\{{i_0\ldots i_{n-1}},{i_1\ldots i_{n}}|{i_0\ldots i_{n}}\in A\right\}
    \end{align}
    as the path complex generated by the path set $A$.
\end{definition}

Now let's consider the generation of path complexes from directed graphs. First, let's clarify the definition of directed graphs and related concepts.
\begin{definition}[Directed Graph]
    A directed graph is a pair $G=\left(V, E\right)$, where the vertex set $V$ is a finite set, and the edge set $E$ is a subset of $V \times V \backslash \Delta_{V}$, where
    \begin{align}
        \Delta_{V}:= \left\{ \left(i, i\right) \in V \times V \mid i \in V \right\}
    \end{align}
    The elements of $\Delta_{V}$ are called self-loops.
\end{definition}

We use $i \rightarrow j\in G$ to denote $(i, j)\in E$, and if it's clear from context, we also write it as $i \rightarrow j$. We use $i \stackrel{\smash{\scriptscriptstyle\rightarrow}}{=} j$ to indicate $i=j$ or $i \rightarrow j$.

\begin{definition}[Undirected Graph]
    An undirected graph is a pair $G=\left(V, E\right)$, where $E$ is a subset of $V \times V \backslash \Delta_{V}$.
\end{definition}

\begin{theorem}[\citet{grigor2012homologies}]
    A path complex generated from a directed graph exists if and only if it satisfies the additional condition: if all $2$-elementary  paths $i_{k-1} i_{k}$ are allowed in the elementary  path $i_{0} \ldots i_{n}$, then the entire elementary  path $i_{0} \ldots i_{n}$ is allowed.
\end{theorem}

We denote by $P\left(G\right)$ the path complex generated by the directed graph $G=(V,E)$. Specifically,
\begin{align}
    P_0\left(G\right) & =V                                                                                                \\
    P_n\left(G\right) & =\left\{i_0\ldots i_{n}| i_0,\ldots, i_{n}\in V, i_{k-1}i_{k} \in E,k=1,\ldots ,n \right\},n\ge 1
\end{align}

Next, let's clearly describe the path complex generated by a simplicial complex.
\begin{definition}[Perfect Path Complex]
    We call a path complex $P$ perfect if any subsequence of allowed elementary  paths of $P$ is also allowed.
\end{definition}
\begin{definition}[Monotonic Path Complex]
    We call a path complex $P$ monotonic if there exists a strictly increasing real-valued function on the vertex set of $P$.
\end{definition}
\begin{theorem}[\citet{grigor2012homologies}]
    A path complex $P$ is generated by some simplicial complex if and only if it is perfect and monotonic.
\end{theorem}

We denote by $P\left(S\right)$ the path complex generated by a simplicial complex $S$. Specifically, by using a monotonic function $f$ from the vertex set of $S$ to reorder and renumber it as $\left\{i_k\right\}_{k=0}^{N}$ such that $\ldots\le f(i_{k-1})\le f(i_{k})\le\ldots$, each simplex of $S$ can then be mapped to a set corresponding to the elementary  paths of $P\left(S\right)$.
\begin{align}
    P\left(S\right)=\left\{i_0\ldots i_{n}| \left\{ i_0,\ldots, i_n\right\}=\sigma\in S \right\}
\end{align}

\subsection{$\partial$-Invariant Paths}
%
Given any path complex $P=\left\{P_{n}\right\}_{n=0}^{\infty}$, where $V$ is a finite vertex set, for any integer $n \geq -1$, consider the $\mathbb{K}$-linear space $\mathcal{A}^{\text{nr}}_{n}$ generated by all elementary  $n$-paths in $P$, i.e.,
\begin{align}
    \mathcal{A}^{\text{nr}}_{n}=\mathcal{A}^{\text{nr}}_{n}\left(P\right)=\operatorname{span}_{\mathbb{K}}\left\{ P_n\right\}
    =\left\{\sum_{i_{0}, \ldots, i_{n} \in V} v^{i_{0} \ldots i_{n}} e_{i_{0} \ldots i_{n}}\mid i_{0} \ldots i_{n} \in P_{n}, v^{i_{0} \ldots i_{n}} \in \mathbb{K}\right\}
\end{align}
For $n=-2$, we define $\mathcal{A}^{\text{nr}}_{-2}=\left\{0\right\}$.

By construction, $\mathcal{A}^{\text{nr}}_{n}$ is a subspace of the space $\Lambda_{n}$. And we have:
\begin{enumerate}
    \item $\mathcal{A}^{\text{nr}}_{0}$ is generated by all the vertices of $P$, thus $\mathcal{A}^{\text{nr}}_{0}=\Lambda_{0}$;
    \item The space $\mathcal{A}^{\text{nr}}_{1}$ is generated by all the edges of $P$, possibly fewer than $\Lambda_{1}$;
    \item $\mathcal{A}^{\text{nr}}_{-1} \cong \mathbb{K}$;
\end{enumerate}
\begin{definition}[Allowed Paths]
    Elements in $\mathcal{A}^{\text{nr}}_{n}$ are called allowed $n$-paths.
\end{definition}

Now we wish to restrict the boundary operator $\partial$ from the space $\Lambda_{n}$ to the space $\mathcal{A}^{\text{nr}}_{n}$. However, generally, $\partial \mathcal{A}^{\text{nr}}_{n}$ is not necessarily a subspace of $\mathcal{A}^{\text{nr}}_{n-1}$. Therefore, for any $n \geq-1$, we consider the following subspace of $\mathcal{A}^{\text{nr}}_{n}$:
\begin{align}
    \Omega^{\text{nr}}_{n}=\Omega^{\text{nr}}\left(P\right)=\left\{v \in \mathcal{A}^{\text{nr}}_{n}: \partial^{\text{nr}} v \in \mathcal{A}^{\text{nr}}_{n-1}\right\}
\end{align}
For $n=-2$, we define $\Omega_{-2}=\left\{0\right\}$.

We always have $\partial^{\text{nr}} \Omega^{\text{nr}}_{n} \subset \Omega^{\text{nr}}_{n-1}$, since if $v \in \Omega^{\text{nr}}_{n}$, then $\partial^{\text{nr}} v \in \mathcal{A}^{\text{nr}}_{n-1}$ and $\partial^{\text{nr}}\left(\partial^{\text{nr}} v\right)=0 \in \mathcal{A}^{\text{nr}}_{n-2}$, hence $\partial^{\text{nr}} v \in \Omega^{\text{nr}}_{n-1}$.
\begin{definition}[$\partial^{\text{nr}}$-Invariant Paths]
    Elements in $\Omega^{\text{nr}}_{n}$ are called $\partial^{\text{nr}}$-invariant $n$-paths.
\end{definition}

Now we consider the regular versions of the above concepts.
\begin{definition}[Regular Path Complex]
    A path complex $P$ is called regular if it does not contain paths of the form $ii$.
\end{definition}
%


\begin{figure}
    \centering
    \includegraphics[width=0.5\linewidth]{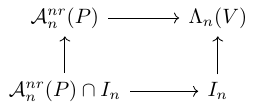}
    \caption{Commutative diagram containing homomorphisms}
    \label{fig:reg}
\end{figure}
\begin{align}
    {\mathcal{A}^{\text{nr}}_{n}\left(P\right)}/{\left\{ \mathcal{A}^{\text{nr}}_{n}\left(P\right) \cap I_{n}\right\}} \rightarrow {\Lambda_{n}\left(V\right)}/{I_{n}}=\mathcal{R}_{n}\left(V\right)
\end{align}

where $n \geq-2$. Denote the image of this homomorphism as
\begin{align}
    \mathcal{A}_{n}\left(P\right) \subset \mathcal{R}_{n}\left(V\right)
\end{align}
Note that ${\mathcal{A}^{\text{nr}}_{n}\left(P\right)}/{\left\{ \mathcal{A}^{\text{nr}}_{n}\left(P\right) \cap I_{n}\right\}}\cong \mathcal{A}_{n}\left(P\right)$.

\begin{definition}[Allowed Regular Paths]
    Elements in $\mathcal{A}_{n}$ are called allowed regular $n$-paths.
\end{definition}
Define $\Omega_{n}\left(P\right) \subset \mathcal{A}_{n}\left(P\right)$ as follows:

\begin{enumerate}
    \item For $n=-2,-1,0$, we define $\Omega_{n}\left(P\right)=\mathcal{A}_{n}\left(P\right)$;
    \item For $n \geq 1$, we define
          \begin{align}
              \Omega_{n}=\Omega_{n}\left(P\right)=\left\{ v \in \mathcal{A}_{n}\left(P\right) \mid \partial v \in \mathcal{A}_{n-1}\left(P\right) \text { where } \partial: \mathcal{R}_{n}\left(V\right) \rightarrow \mathcal{R}_{n-1}\left(V\right)\right\}
          \end{align}
\end{enumerate}

\begin{definition}[$\partial$-Invariant Paths]
    Elements in $\Omega_{n}$ are called $\partial$-invariant $n$-paths.
\end{definition}
%


We obtain the chain complex $\Omega_*\left(P\right)$ of path homology for the path complex $P$:
\begin{align}
    0 \leftarrow \Omega_{0} \leftarrow \ldots \leftarrow \Omega_{n-1} \leftarrow \Omega_{n}\leftarrow \Omega_{n+1}\leftarrow \ldots
\end{align}

\begin{definition}[Homology Groups]
    The homology groups of $\Omega_*\left(P\right)$ are called the path homology groups of the path complex $P$, denoted as $H_*\left(P\right)=\left\{H_{n}\left(P\right)\right\}_{n \geq 0}$.
\end{definition}

If $P\left(G\right)$ is the path complex generated by the directed graph $G$, then we use the notation
\begin{align}
    \Omega_{n}\left(G\right):=\Omega_{n}\left(P\left(G\right)\right)
\end{align}

The corresponding homology groups will be denoted as $H_{n}\left(G\right)$ and called the path homology groups of the directed graph $G$.

It is also noted that one can determine $H_{n}$ using only $\mathcal{A}_{n},\mathcal{A}_{n+1}$.
\begin{proposition}[\citet{grigor2012homologies}]
    \label{determine}
    We have:
    \begin{align}
        H_{n}=\left.\operatorname{Ker} \partial\right|_{\mathcal{A}_{n}} /\left(\mathcal{A}_{n} \cap \partial \mathcal{A}_{n+1}\right),\forall n\ge 0
    \end{align}
\end{proposition}

\section{Homotopy of Directed Graphs and Path Complexes}

\subsection{Induced Maps}
We refer to \citet{lin2019weighted} for the definition and properties of induced maps.

Given two finite sets $V$ and $V^{\prime}$ and a mapping $f: V \rightarrow V^{\prime}$ between them, we define the induced map $f_{*}: \Lambda_{*}\left(V\right) \rightarrow \Lambda_{*}\left(V^{\prime}\right)$ on the path chain complex $\Lambda_*$ as follows:
\begin{enumerate}
    \item For $p =-1,-2$, $f_{*}$ is defined as an isomorphism.
    \item For $p \geq 0$, $f_{*}$ is defined on the natural basis as $f_{*}\left(e_{i_{0} \ldots i_{p}}\right)=e_{f\left(i_{0}\right) \ldots f\left(i_{p}\right)}$, and then extended $\mathbb{K}$-linearly to all elements of $\Lambda_{p}\left(V\right)$.
\end{enumerate}

Let $\partial^\text{nr},\partial^{\text{nr} \prime} $ denote the non-regular boundary operators on $\Lambda_{*}\left(V\right),\Lambda_{*}\left(V^{\prime}\right)$ respectively. It can be verified by direct calculation that
\begin{align}
    \partial^{\text{nr} \prime} f_*=f_* \partial^{\text{nr} }
\end{align}
Thus, $f_*: \Lambda_{*}\left(V\right) \rightarrow \Lambda_{*}\left(V^{\prime}\right)$ is a chain map.

If $e_{i_{0} \ldots i_{p}}$ is non-regular, then $f_{*}\left(e_{i_{0} \ldots i_{p}}\right)$ is also non-regular, hence
\begin{align}
    f_{*}\left(I_{p}\left(V\right)\right) \subset I_{p}\left(V^{\prime}\right)
\end{align}
Thus, $f_{*}$ is well-defined on the quotient $\Lambda_{p} / I_{p}$, and we can further induce a map on the regular chain complex $\mathcal{R}_*$:
\begin{align}
    f_{*}: \mathcal{R}_{*}\left(V\right) \rightarrow \mathcal{R}_{*}\left(V^{\prime}\right)
\end{align}

From the definition of $\mathcal{R}_{p}$, we derive the rule for $f_{*}: \mathcal{R}_{p}\left(V\right) \rightarrow \mathcal{R}_{p}\left(V^{\prime}\right)$ as follows:
\begin{align}
    f_{*}\left(e_{i_{0} \ldots i_{p}}\right)= \begin{cases}e_{f\left(i_{0}\right) \ldots f\left(i_{p}\right)} & \text {if } e_{f\left(i_{0}\right) \ldots f\left(i_{p}\right)} \text { is regular} \\ 0, & \text {otherwise }\end{cases}
\end{align}

Let $\partial,\partial^{\prime} $ denote the regular boundary operators on $\mathcal{R}_{*}\left(V\right),\mathcal{R}_{*}\left(V^{\prime}\right)$ respectively. It can be verified by direct calculation that
\begin{align}
    \partial^{\prime} f_*=f_* \partial
\end{align}
Thus, $f_*: \mathcal{R}_{*}\left(V\right) \rightarrow \mathcal{R}_{*}\left(V^{\prime}\right)$ is a chain map.

\subsection{Homotopy of Directed Graphs}

For directed graphs, we consider their sets of vertices and apply the concepts introduced above.

\begin{definition}[Directed Graph Map]
    If there exists a map $f: V_{G} \rightarrow V_{G^\prime}$ on the vertices such that
    \begin{align}
        i \rightarrow j\in G \Longrightarrow f\left(i\right) \stackrel{\smash{\scriptscriptstyle\rightarrow}}{=} f\left(j\right) \in G^\prime
    \end{align}
    we call $f$ an induced map between directed graphs $G$ and $G^\prime$, denoted as $f: G \rightarrow G^\prime$. Composition of directed graph maps $f,g$ is defined as the composition of vertex maps, inducing a map on directed graphs. We denote the category of all directed graphs and induced maps by $\mathbf{Dgr}$.
\end{definition}

$\mathbf{Dgr}$ satisfies the category axioms, see \citet{grigor2014homotopy}.

\begin{theorem}[\citet{grigor2014homotopy}]
    Let $G$ and $G^{\prime}$ be two directed graphs, and $f: G \rightarrow G^{\prime}$ be a directed graph map. Then the map $\left.f_{*}\right|_{\Omega_{p}\left(G\right)}$ provides a chain map on the path homology chain complex:
    \begin{align}
        f_{*}:\Omega_{*}\left(G\right) \rightarrow \Omega_{*}\left(G^{\prime}\right)
    \end{align}
    Thus, we also obtain group homomorphisms on path homology groups:
    \begin{align}
        f_{*}: H_{*}\left(G\right) \rightarrow H_{*}\left(G^{\prime}\right)
    \end{align}
\end{theorem}

Let $G_{X}=\left(X, E_{X}\right), G_{Y}=\left(Y, E_{Y}\right)$ be two directed graphs. The product directed graph $G_{X} \times G_{Y}=\left(X \times Y, E_{X \times Y}\right)$ is defined as follows:
\begin{align}
    X \times Y     & := \left\{ \left(x, y\right)\mid x \in X, y \in \right\}                                                                                                                                                        \\
    E_{X \times Y} & :=\left\{ \left(\left(x, y\right),\left(x^{\prime}, y^{\prime}\right)\right) \in\left(X \times Y\right)^{2}\mid x=x^{\prime},y\rightarrow y^{\prime} \text { or } y=y^{\prime},x\rightarrow x^{\prime} \right\}
\end{align}

A directed graph $I$ is defined as a directed graph with vertices $ \left\{ 0,1 \right\} $ and an edge $\left(0,1\right)$.

\begin{definition}[One-Step Homotopy]
    Two directed graph maps $f, g: G_{X} \rightarrow G_{Y}$ are said to be one-step homotopic, denoted by $f \simeq_1 g$, if there exists a directed graph map $F: G_{X} \times I \rightarrow G_{Y}$ such that
    \begin{align}
        \left.F\right|_{G_{X} \times \left\{ 0 \right\} }=f,\left.F\right|_{G_{X} \times \left\{ 1 \right\} }=g \text { or } \left.F\right|_{G_{X} \times \left\{ 0 \right\} }=g,\left.F\right|_{G_{X} \times \left\{ 1 \right\} }=f
    \end{align}
    This is equivalent to
    \begin{align}
        f\left(x\right) \stackrel{\smash{\scriptscriptstyle\rightarrow}}{=} g\left(x\right), \forall x \in X \text { or } g\left(x\right) \stackrel{\smash{\scriptscriptstyle\rightarrow}}{=} f\left(x\right), \forall x \in X
    \end{align}
\end{definition}
\begin{definition}[Homotopy]
    Two directed graph maps $f, g: G_{X} \rightarrow G_{Y}$ are said to be homotopic, denoted by $f \simeq g$, if there exist directed graph maps $f_{0}=f, f_{1}, \ldots, f_{n}=g: G_{X} \rightarrow G_{Y}$ such that $f_{i-1}, f_{i}, i=0,\ldots, n$ are one-step homotopic.
\end{definition}

\begin{definition}[Homotopy Equivalence]
    Directed graphs $G_{X}$ and $G_{Y}$ are said to be homotopy equivalent if there exist directed graph maps $f: G_{X} \rightarrow G_{Y}$ and $g: G_{Y} \rightarrow G_{X}$ such that $g \circ f \simeq \operatorname{id}_{G_{X}}$ and $f \circ g \simeq \operatorname{id}_{G_{Y}}$.
\end{definition}

From the concept of homotopy, the following theorem about the homotopy groups of paths can be derived:

\begin{theorem}[\citet{grigor2014homotopy}]
    \label{homotopy1}
    For directed graphs $G, G^{\prime}$ and homotopic directed graph maps $f, g: G \rightarrow G^{\prime}$, the induced maps on homology groups are the same, i.e.:
    \begin{align}
        f_{ * }=g_{ * }: H_{*}\left(G\right) \rightarrow H_{*}\left(G^{\prime}\right)
    \end{align}
    If the digraphs $G$ and $H$ are homotopy equivalent, then they have isomorphic homology groups. Furthermore, if the homotopical equivalence of G and H is provided by the digraph maps then their induced maps $f_*$ and $g_*$ provide mutually inverse isomorphisms of the homology groups of $G$ and $H$.

\end{theorem}

\subsection{Homotopy of Path Complexes}
Now we consider the case of path complexes. Since mappings on vertex sets are often considered in this subsection, we represent a path complex $P$ with the tuple $(V,P)$ to emphasize its vertex set $V$.

As discussed earlier, for two finite sets $V, V^{\prime}$ and mappings $f: V \rightarrow V^{\prime}$, we denote by $f_{*}$ the induced map $f_{*}: \Lambda_{*}\left(V\right) \rightarrow \Lambda_{*}\left(V^{\prime}\right)$.

\begin{definition}[Weak Morphism of Path Complexes]
    If for any path $v \in P$, the path $f_{*}\left(v\right)$ either lies in $P^{\prime}$ or is non-regular, then the set map $f: V \rightarrow V^{\prime}$ is called a weak morphism between path complexes $P$ and $P^{\prime}$, denoted
    \begin{align}
        f_{\circ}=(f,f_*):\left(V, P\right) \rightarrow\left(V^{\prime}, P^{\prime}\right)
    \end{align}
    For weak morphisms of path complexes $ f_{\circ}, g_{\circ}=(g,g_*)$, their composition is defined as $g_{\circ}\circ f_{\circ}:=\left(g\circ f, g_*\circ f_*\right)$.
    We denote by $\mathbf{PC}$ the category consisting of all path complexes and weak morphisms between them.
\end{definition}

$\mathbf{PC}$ satisfies the category axioms, see \citet{carlsson2009topology}.

We denote by $\mathbf{Ch}$ the category of chain complexes, where objects are chain complexes and morphisms are chain maps.

\begin{theorem}[\citet{grigor2019homology}]
    The weak morphism $f_{\circ}:\left(V, P\right) \rightarrow\left(V^{\prime}, P^{\prime}\right)$ of simplicial complexes induces chain maps on the chain complexes of simplicial complexes:
    \begin{align}
        f_{*}: \Omega_{*}\left(P\right) \rightarrow \Omega_{*}\left(P^{\prime}\right)
    \end{align}
    and homomorphisms on the homology groups:
    \begin{align}
        f_{*}: H_{*}\left(P\right) \rightarrow H_{*}\left(P^{\prime}\right)
    \end{align}
\end{theorem}
\begin{proposition}[\citet{grigor2019homology}]
    For weak morphism $f_{\circ}:\left(V, P\right) \rightarrow\left(V^{\prime}, P^{\prime}\right)$ between simplicial complexes, we define
    \begin{align}
        \Omega_{*}\left(f_{\circ}\right):=f_{*}: \Omega_{*}\left(P\right) \rightarrow \Omega_{*}\left(P^{\prime}\right)
    \end{align}
    Then $\Omega_{*}:\mathbf{PC}\to \mathbf{Ch}$ is a functor.
    \label{fun_pc_ch}
\end{proposition}

Now we begin the introduction to homotopy theory of simplicial complexes.

Let $I= \left\{ 0,1 \right\} $ be the set containing two elements. For any set $V= \left\{ 0, \ldots, n \right\} $, let $V \times I$ denote the Cartesian product. Let $V^{\prime}$ be a copy of the set $V$, denoted by $V^{\prime}=\left\{ 0^{\prime}, \ldots, n^{\prime}\right\}$, where $i^{\prime} \in V^{\prime}$ corresponds to $i \in V$. Then, we can associate $V \times I$ with $V \cup V^{\prime}$, where for any $i \in V$, $\left(i, 0\right)$ corresponds to $i$, and $\left(i, 1\right)$ corresponds to $i^{\prime}$. Thus, $V$ is associated with $V \times \left\{ 0 \right\} \subset V \times I$, and $V^{\prime}$ is associated with $V \times \left\{ 1 \right\} \subset V \times I$.

We define a simplicial complex $P^{\prime}$ on the set $V^{\prime}$ by the following condition:
\begin{align}
    i_{0}^{\prime} \ldots i_{n}^{\prime} \in P^{\prime}\iff i_{0} \ldots i_{n} \in P
\end{align}

Define the simplicial complex $P \times I$ as a simplicial complex on $V \times I=V \cup V^{\prime}$:
\begin{align}
    P \times I
    = & \left\{ w \mid w \in P \right\}
    \cup\left\{ w^{\prime} \mid w^{\prime} \in P^{\prime}\right\} \notag                                                                           \\
      & \cup\left\{ \widehat{w}=i_{0} \ldots i_{k} i_{k}^{\prime} \ldots i_{n}^{\prime} \mid i_{0} \ldots i_{k} i_{k+1} \ldots i_{n} \in P\right\}
\end{align}
where $0 \leq k \leq n$.
\begin{definition}[Weak One-Step Homotopy]
    We call two weak morphisms $f_{\circ}, g_{\circ}:\left(V, P\right) \rightarrow\left(W, S\right)$ of simplicial complexes weakly one-step homotopic, denoted as $ f_{\circ} \sim_{1} g_{\circ}$, if there exists a weak morphism $F_{\circ}:\left(V \times I, P \times I\right) \rightarrow\left(W, S\right)$ of simplicial complexes and at least one of the following two conditions holds:
    \begin{align}
        \begin{aligned}
             & \text { 1. }\left.F_{\circ}\right|_{\left(V, P\right)}=f_{\circ},\left.F_{\circ}\right|_{\left(V^{\prime}, P^{\prime}\right)}=g_{\circ}^{\prime} \\
             & \text { 2. }\left.F_{\circ}\right|_{\left(V, P\right)}=g_{\circ},\left.F_{\circ}\right|_{\left(V^{\prime}, P^{\prime}\right)}=f_{\circ}^{\prime}
        \end{aligned}
    \end{align}
\end{definition}
\begin{definition}[Weak Homotopy]
    We call two weak morphisms $f_{\circ}, g_{\circ}:\left(V, P\right) \rightarrow\left(W, S\right)$ of simplicial complexes weakly homotopic, denoted as $f_{\circ} \sim g_{\circ}$, if there exists a sequence of weak morphisms
    \begin{align}
        f_{i_{\circ}}:\left(V, P\right) \rightarrow\left(W, S\right)
    \end{align}
    such that $f_{\circ}=f_{0 \circ} \sim_{1} f_{1 \circ} \sim_{1} \cdots \sim_{1} f_{n_{\circ}}=g_{\circ}$.
\end{definition}

\begin{definition}[Weak Homotopy Equivalence]
    If there exist weak morphisms
    \begin{align}
        f_{\circ}:\left(V, P\right) \rightarrow\left(W, S\right), g_{\circ}:\left(W, S\right) \rightarrow\left(V, P\right)
    \end{align}
    such that
    \begin{align}
        f_{\circ}\circ g_{\circ} \sim \operatorname{id}_{W \circ}, g_{\circ
                \circ} \sim \operatorname{id}_{V \circ}
    \end{align}
    then two simplicial complexes $\left(V, P\right)$ and $\left(W, S\right)$ are weakly homotopy equivalent.
 
\end{definition}

\begin{theorem}[\citet{grigor2019homology}]
        Weakly homotopic morphisms of simplicial complexes
    \begin{align}
        f_{\circ} \sim g_{\circ}:\left(V, P\right) \rightarrow\left(W, S\right)
    \end{align}
    induce homotopy chain maps on chain complexes
    \begin{align}
        f_{*} \simeq g_{*}: \Omega_{*}\left(P\right) \rightarrow \Omega_{*}\left(S\right)
    \end{align}
    thereby inducing identical homomorphisms on the homology groups
    \begin{align}
        f_{*} = g_{*}: H_{*}\left(P\right) \rightarrow H_{*}\left(S\right)
    \end{align}
    If the path complexes $\left(V, P\right)$ and $\left(W, S\right)$ are weakly homotopy equivalent, then they have isomorphic homology groups. Furthermore, if the weak homotopy equivalence is provided by $f_{\circ}$ and $g_{\circ}$, then the induced group homomorphisms $f_{*}$ and $g_{*}$ provide mutual inverse isomorphisms of the homology groups of $\left(V, P\right)$ and $\left(W, S\right)$.
    \label{homotopy2}
\end{theorem}

\begin{proposition}
    For a directed graph map $f\in \operatorname{Mor}_{\mathbf{Dgr}}[G,H] $, let $P(f):=f_\circ$, then $P:\mathbf{Dgr}\to \mathbf{PC} $ is a functor.
\end{proposition}
\begin{proof}
    For a directed graph map $f\in \operatorname{Mor}_{\mathbf{Dgr}}[G,H] $ and $p=e_{i_0\ldots i_n }\in P(G)$, we have $f(i_{k-1}) \stackrel{\smash{\scriptscriptstyle\rightarrow}}{=} f(i_k)\in H, k=1,\ldots,n$, thus either $f_*(p)\in P(H)$ or $f_*(p)$ is irregular. Hence, $P(f)\in \operatorname{Mor}_{\mathbf{PC}}[P(G),P(H)]$.
\end{proof}

\section{Persistent Modules and Persistent Homology}

\subsection{Persistent Diagram}

We use $\mathbb{R}$ as the index set.

We denote $\overline{\mathbb{R}}$ as the extended real numbers $\left[-\infty, +\infty\right]$.

We use $\mathbf{R}$ to denote the partially ordered set $\mathbb{R}$ equipped with the $\leq$ relation, and consider it as a category.

\begin{definition}[Persistent Vector Space]
    We denote by ${}_{\mathbb{K}} \mathbf{Vec}$ the category of $\mathbb{K}$-vector spaces. A persistent vector space is a functor $\mathbf{R} \rightarrow {}_{\mathbb{K}} \mathbf{Vec}$.
    The morphisms of persistent vector spaces are morphisms in the category $\left[\mathbf{R} , {}_{\mathbb{K}} \mathbf{Vec} \right]$. That is, for $\mathcal{M}, \mathcal{N} \in\left[\mathbf{R} , {}_{\mathbb{K}} \mathbf{Vec} \right]$, a morphism $\phi: \mathcal{M} \rightarrow \mathcal{N}$ consists of a family of linear maps $\left\{ \phi_{t}: \mathcal{M}\left(t\right) \rightarrow \mathcal{N}\left(t\right)\right\} $ such that
    \begin{align}
        \mathcal{N}\left(s \leq t\right) \circ \phi_{s}=\phi_{t} \circ \mathcal{M}\left(s \leq t\right)
    \end{align}
    holds for all $s \leq t$.

    We denote the category of persistent vector spaces as ${}_{\mathbb{K}} \mathbf{PersVec} :=\left[\mathbf{R}, {}_{\mathbb{K}} \mathbf{Vec}\right]$.
\end{definition}

In this paper, unless otherwise stated, persistent modules are indexed by $\mathbb{N}$, while persistent vector spaces are indexed by $\mathbb{R}$.

\begin{definition}[Finite Persistent Vector Space]
    We call a persistent vector space $ \mathcal{M}: \mathbf{R} \rightarrow {}_{\mathbb{K}}  \mathbf{Vec}$ finite if $ \mathcal{M}\left(t\right)$ is finite-dimensional for all $t \in \mathbb{R}$ and there exist only finitely many $t \in \mathbb{R}$ such that $ \mathcal{M}\left(t-\epsilon \leq t+\epsilon\right)$ is not isomorphic for any $\epsilon>0$. We call these $t$ critical values.

    We denote the full subcategory of ${}_{\mathbb{K}} \mathbf{PersVec}$ consisting of finite persistent vector spaces as ${}_{\mathbb{K}} \mathbf{Persvec}$.
\end{definition}

\begin{definition}[Interval Vector Space]
    For an interval $I \subseteq R$, we define the corresponding interval vector space $P_{\mathbb{K}} \left(I\right) \in {}_{\mathbb{K}} \mathbf{Persvec}$ as:
    \begin{align}
        P_{\mathbb{K}} \left(I\right)\left(t\right):=\begin{cases}\mathbb{K} & \text { if } t \in I \\0 & \text { otherwise}
                                                     \end{cases}
    \end{align}
    If $s, t \in I$, then $P_{\mathbb{K}} \left(I\right)\left(s \leq t\right)$ is the identity map; otherwise, it is the trivial map.
\end{definition}

We note that when a persistent vector space is finite, there exist only finitely many critical values $t_{0}<\cdots<t_{k}$, such that if $t_{i-1}<s \leq t<t_{i}$, then $ \mathcal{M}\left(s \leq t\right)$ is isomorphic. Therefore, all information of finite persistent vector spaces is contained in the maps $ \mathcal{M}\left(t_{i-1} \leq t_{i}\right)$ for $i=1, \ldots, k$. Thus, we can consider it as a functor $\left[k\right] \rightarrow{}_{\mathbb{K}}  \mathbf{{Vec}}$, where $\left[k\right]$ is the partially ordered subset of $(\mathbb{N},\le )$ consisting of the integers $0, \ldots, k$. In this case, we can view it as a finite persistent vector space.

Therefore, under the finite condition, the correspondences and structure theorems mentioned previously for the discrete version can be similarly extended to the continuous version and summarized as follows:
\begin{theorem}[\citet{chazal2016structure}]
    Given $\mathcal{M} \in{}_{\mathbb{K}} \mathbf{Persvec}$, there exists a multiset $B(\mathcal{M})$ of intervals in $\mathbb{R}$, such that
    \begin{align}
        \mathcal{M} \cong \bigoplus_{I \in B(\mathcal{M}) } P_{\mathbb{K}} \left(I\right)
    \end{align}
    and this decomposition is unique up to ordering.
\end{theorem}

Similarly, we consider the continuous version of the concept of characteristic intervals mentioned earlier.
\begin{definition}[Persistent Barcode]
    We call the multiset $B(\mathcal{M})$ given by the above theorem the persistent barcode of the finite persistent vector space $\mathcal{M}$, which consists of a multiset of intervals in $\mathbb{R}$.
\end{definition}

We can view intervals in $\mathbb{R}$ as points in $\overline{\mathbb{R}}^2$, leading to the following description.
\begin{definition}[Persistent Diagram]
    Given $\mathcal{M} \in{}_{\mathbb{K}} \mathbf{Persvec}$, we regard its persistent barcodes as a multiset in $\overline{\mathbb{R}}^2$, and call it a persistent diagram.
    \begin{align}
        D\left(\mathcal{M}\right):=\left\{ \left\{ \left(a_{k}, b_{k}\right)\in \overline{\mathbb{R}}^2 \mid I_{k} \in B(\mathcal{M}) \text { with endpoints } a_{k} \leq b_{k}\right\} \right\}.
    \end{align}

\end{definition}

\subsection{Filtration}

The complex networks we consider are often represented by weighted directed graphs.

\begin{definition}[Weighted Directed Graph]
    A weighted directed graph is a triple $G=\left(V, E, w\right)$, where $\left(V, E\right)$ is a directed graph and $w: E \rightarrow \mathbb{R}_{>0}$ is a positive function on edges.
\end{definition}

Given a weighted directed graph $G$, we can consider the simplicial complex $P(G)$ generated by the directed graph $G$, while retaining the weights of the edges of the weighted directed graph $G$. Thus, we obtain a simplicial complex $P(G)$ with edge weights.

However, in practical applications, we may not be interested in all the higher-order structures in $P(G)$. For example, if we only want to deal with structures of order up to 3, we should consider $P(G)_0\cup P(G)_1\cup P(G)_2 $ instead of $P(G) $ . Also, if we are only interested in certain specific forms of structures, we may make finer selections of paths in $P(G)$. Moreover, the representations of some complex networks that contain higher-order structures themselves are simplicial complexes rather than directed graphs. Therefore, we should also consider methods to establish persistent path homology given an edge-weighted simplicial complex $P(G)$.

Given a weighted directed graph, or more generally, given an edge-weighted simplicial complex, to establish persistent path homology, we use the weights to construct a scale parameter and generate a family of simplicial complexes indexed by the scale parameter using the concept of a filtration map.

We use the prefix $\mathbf{W}$ to denote the respective category with edge weights, where morphisms are morphisms in the underlying category without edge weights. For example, $\mathbf{WDgr}$ denotes the category of weighted directed graphs, where morphisms are directed graph mappings. Similarly, $\mathbf{WPC}$ denotes the category of edge-weighted simplicial complexes, where morphisms are weak maps of simplicial complexes.

\begin{definition}[Filtration Map]
    A filtration map is a map $F: \operatorname{Obj}\left(\mathbf{WDgr}\right) \rightarrow \operatorname{Obj}\left(\left[\mathbf{R}, \mathbf{Dgr}\right]\right)$ or $F: \operatorname{Obj}\left(\mathbf{WPC}\right) \rightarrow \operatorname{Obj}\left(\left[\mathbf{R}, \mathbf{PC}\right]\right)$ such that $V\left(F^{t} S\right) \subseteq V\left(S\right)$ holds for all $t \in \mathbb{R}$.

    We denote $F^{t} S:=F\left(S\right)\left(t\right)$ as the image of $t$ under the functor $F\left(S\right)$, and $F\left(S\right)\left(s \leq t\right)$ as the image of $s \leq t$ under the functor $F\left(S\right)$.
\end{definition}

Below, we present several filtration maps considered in this paper.

\begin{definition}[Edge Level Filtration]
    Given a weighted directed graph
    \begin{align}
        G=\left\{V,E,w\right\}
    \end{align}
    Define $G^\delta=\left\{V^\delta,E^\delta,w^\delta\right\},\forall \delta\in \mathbb{R}$ as follows.
    \begin{align}
        V^\delta & = V \notag                                                 \\
        E^\delta & = \left\{e\in E| w\left(e\right)\le \delta \right\} \notag \\
        w^\delta & = w|_{E^\delta}
    \end{align}
    Also, denote $P^\delta=P\left(G^\delta\right)$. We denote this filtration map as $F_e:G\mapsto \left\{G^\delta\right\}_{\delta\in \mathbb{R}}$.
\end{definition}

\begin{definition}[Length of a Path]
    Given fixed edges $E$ and their edge weights $w:E\to \mathbb{R}_{>0} $, and a path $p=i_0\ldots i_n$, the length of $p$ is defined as
    \begin{align}
        \operatorname{len}\left(p\right):=\sum_{i=0}^{n} w\left(e_{i}\right)
    \end{align}
\end{definition}

\begin{definition}[Path Sublevel Filtration]
    Given an edge-weighted path complex
    \begin{align}
        P=\left\{\left\{P_n\right\}_{n=0}^\infty,w:P_1\to \mathbb{R}_{>0}\right\}
    \end{align}
    consider defining
    \begin{align}
        P^\delta=\left\{V^\delta=P_0^\delta,E^\delta=P_1^\delta,\left\{P_n^\delta\right\}_{n=0}^\infty,w^\delta\right\},\forall \delta\in \mathbb{R}
    \end{align}
    \begin{align}
        P_0^\delta & = P_0^\delta \notag                                                                 \\
        P_n^\delta & = \left\{p\in P_n|\operatorname{len}\left(P\right)\le \delta \right\},n\ge 1 \notag \\
        w^\delta   & = w|_{P_1^\delta}
    \end{align}
    It can be directly verified that $P^\delta\subset P\left(P_1^\delta\right) $, and $P^\delta$ indeed forms a path complex. We denote this filtration map as $F_p:P\mapsto \left\{P^\delta \right\}_{\delta\in \mathbb{R}}$.
\end{definition}

\subsection{Persistent Path Homology}
Now we establish persistent path homology.

Recalling that given a chain complex $C_* \in \mathbf{Ch}$, we use $C_{k}$ to denote the $k$th chain group, and $\partial_{k}: C_{k} \rightarrow C_{k-1}$ to denote the boundary operator. For $k \in \mathbb{N}$, the $k$th homology group $H_{k}\left(C_*\right)$ is the quotient
\begin{align}
    H_{k}\left(C_*\right):={\operatorname{Ker} \partial_{k}}/{\operatorname{Im} \partial_{k+1}}
\end{align}

Considering chain complexes $C_*$ and $C_*^\prime$, equipped with boundary operators $\partial,\partial^\prime $ and having $k$th homology groups $H_k(C_*),H_k(C_*^\prime)$, for a chain map $f:C_*\to C_*^\prime$, we also define
\begin{align}
    H_k(f):=f_k|_{H_k(C_*)}
\end{align}
\begin{proposition}[\citet{hatcher2005algebraic} ]
    $H_{k}:\mathbf{Ch} \rightarrow\mathbf{Vec}$ is a functor.
    \label{fun_ch_vec}
\end{proposition}

We recall the pineline of persistent path homology under the seting of weighted digraph:
\begin{align}
    \mathbf{WDgr}\xrightarrow{F}\left[\mathbf{R}, \mathbf{Dgr}\right]\xrightarrow{\left[\mathbf{R}, P\right]}\left[\mathbf{R}, \mathbf{PC}\right] \xrightarrow{\left[\mathbf{R}, \Omega_*\right]}\left[\mathbf{R}, \mathbf{Ch}\right] \xrightarrow{\left[\mathbf{R}, H_{k}\right]}\left[\mathbf{R}, \mathbf{Vec}\right]\xrightarrow{D} \mathbf{Mult}\left(\overline{\mathbb{R}}^2\right)
\end{align}

Here we denote $D_k=D\circ\left[\mathbf{R}, H_{k}\right]\circ\left[\mathbf{R}, \Omega_*\right]\circ\left[\mathbf{R}, P\right]\circ F$.

For the case of a given edge-weighted path complex, we can skip the step of generating the edge-weighted path complex from the weighted directed graph:
\begin{align}
    \mathbf{WPC}\xrightarrow{F}\left[\mathbf{R}, \mathbf{PC}\right] \xrightarrow{\left[\mathbf{R}, \Omega_*\right]}\left[\mathbf{R}, \mathbf{Ch}\right] \xrightarrow{\left[\mathbf{R}, H_{k}\right]}\left[\mathbf{R}, \mathbf{Vec}\right]\xrightarrow{D} \mathbf{Mult}\left(\overline{\mathbb{R}}^2\right)
\end{align}

Here we denote $D_k=D\circ\left[\mathbf{R}, H_{k}\right]\circ\left[\mathbf{R}, \Omega_*\right]\circ F$.

Note that both the directed graphs and path complexes in this paper require that the vertex set $V$ be finite, so for any $n\in \mathbb{N}$, $\Lambda_n(V)$ is finite. Therefore, for any filtration map introduced in the previous subsection, the scale parameter is essentially a finite set, ensuring that the resulting persistent vector space is finite, thereby ensuring the calculation of persistent barcodes.

\section{Stability}

If we want to apply persistent homology to real data and complex networks, we naturally hope that persistent graphs, as summary quantities, are stable with respect to noise and disturbances in the input data. To achieve this, we first introduce the distance of persistent graphs and the interleaving property of persistent vector spaces. Then, we combine the homotopy theory of directed graphs and road complexes to establish stability results of persistent graphs with respect to the weights and structures of complex networks on weighted directed graphs and edge-weighted road complexes. This is also the main result of this paper.

\subsection{Bottleneck Distance}
To quantify stability, we need to establish a measure between persistence diagrams. In this paper, we use the bottleneck distance as mentioned by \citet{botnantopological}.

Let $D_1$ and $D_2$ be multisets of points in $\overline{\mathbb{R}}^{2}$. We define a matching between $D_1$ and $D_2$ as $\chi= \left\{ \left(I, J\right) \in D_1 \times D_2 \right\}$, where each $I$ and each $J$ can appear in at most one pair. Similarly, a matching is a bijection between a subset of $D_1$ and a subset of $D_2$. If $\left(I, J\right) \in \chi$, we say $I$ matches with $J$. If $I$ does not appear in such a pair, we say $I$ does not match.

The matching cost $c\left(I, J\right)$ between $I=\left( a, b\right)$ and $J=\left( c, d\right)$ is defined as
\begin{align}
    c\left(I, J\right)=\max \left(|c-a|,|d-b|\right),
\end{align}
and the cost $c\left(I\right)$ for unmatched $I$ is defined as
\begin{align}
    c\left(I\right)=\left(b-a\right) / 2
\end{align}
Given this, we define the cost of matching $\chi$ as
\begin{align}
    c\left(\chi\right):=\max \left(\sup _{\left(I, J\right) \in \chi} c\left(I, J\right), \sup _{\text {unmatched } I \in D_1 \cup D_2} c\left(I\right)\right)
\end{align}
If $c\left(\chi\right) \leq \epsilon$, we say $\chi$ is an $\epsilon$-matching.

A geometric interpretation of this is as follows: an $\epsilon$-matching pairs points $p$ and $q$, corresponding to elements in $D_1$ and $D_2$ respectively, such that $\|p-q\|_{\infty} \leq \epsilon$, and ensures that any unmatched point is at most $\epsilon$ away from the diagonal in the $l^{\infty}$ norm. From a matching perspective, we can define the following distance for multisets in extended $\mathbb{R}^2$.
\begin{definition}[Bottleneck Distance]
    The bottleneck distance between multisets $D_1$ and $D_2$ is
    \begin{align}
        d_{B}\left(D_1, D_2\right)=\inf \left\{ c\left(\chi\right): \chi \text { is a matching between } D_1 \text { and } D_2 \right\}
    \end{align}
\end{definition}

By merging multisets with infinitely many repetitions of the diagonal, and considering matching as bijections between multisets, we can give the following equivalent definition:

\begin{definition}[Bottleneck Distance]
    Given two multisets $D_{1}, D_{2} \subseteq \overline{\mathbb{R}}^{2}$, the bottleneck distance is
    \begin{align}
        d_{B}\left(D_{1}, D_{2}\right):=\inf _{\gamma} \sup _{x \in D_{1} \cup \Delta^\infty}\|x-\gamma\left(x\right)\|_{\infty}
    \end{align}
    where $\gamma$ ranges over all multisets bijections $D_{1} \cup \Delta^\infty \rightarrow D_{2} \cup \Delta^\infty$, $\Delta= \left\{ \left(x, x\right) \mid x \in \overline{\mathbb{R}} \right\} $ is the diagonal, and $\Delta^\infty$ is the diagonal with infinitely many repetitions.
\end{definition}

\subsection{Interleaving and Stability of Persistent Vector Spaces}
The bottleneck distance involves bijections between multisets, which are difficult to control, and the computation of persistence diagrams itself involves the decomposition of homology groups, making it computationally complex. Here we refer to the interleaving technique used by \citet{chazal2009proximity}, which transforms into considering homomorphisms between homology groups.
\begin{definition}[Interleaving]
    For persistent vector spaces $\mathcal{F}$ and $\mathcal{G}$, if there exist families of morphisms
    \begin{align}
        u_{\alpha}: F_{\alpha} \rightarrow G_{\alpha+\epsilon} ,\alpha \in \mathbb{R} \\
        v_{\alpha}: G_{\alpha} \rightarrow F_{\alpha+\epsilon} ,\alpha \in \mathbb{R}
    \end{align}
    such that \ref{fig:sil} commutes for all $\alpha \leq \alpha^{\prime} \in \mathbb{R}$, then these two persistent vector spaces are said to be $\epsilon$-interleaved.
\end{definition}
\begin{figure}
    \centering
    \includegraphics[width=0.8\linewidth]{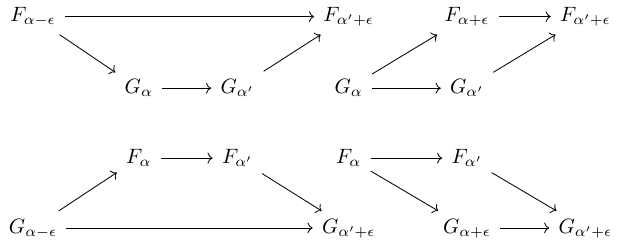}
    \caption{Commutative diagrams to be satisfied for interleaving}
    \label{fig:sil}
\end{figure}

Based on the concept of interleaving, \citet{chazal2009proximity} defined the following distance for persistent vector spaces.
\begin{definition}
    The interleaving distance between two persistent vector spaces $\mathcal{F}$ and $\mathcal{G}$ is
    \begin{align}
        d_{I}\left(\mathcal{F}, \mathcal{G}\right)=\inf \left\{ \epsilon |\mathcal{F}, \mathcal{G}\  \epsilon \text {-interleaved} \right\}
    \end{align}
\end{definition}

Based on this, we have an important algebraic stability theorem in persistent homology.
\begin{theorem}[\citet{bauer2014induced}]
    \label{stab}
    For finite persistent vector spaces $\mathcal{F}$ and $\mathcal{G}$, we have
    \begin{align}
        d_{B}\left(D (\mathcal{F}), D (\mathcal{G})\right)=d_{I}\left(\mathcal{F}, \mathcal{G}\right)
    \end{align}
\end{theorem}

\subsection{Stability Results Based on Weighted Directed Graphs}

\begin{definition}
    For weighted directed graphs $G$ and $H$, and a directed graph mapping $\varphi: G\rightarrow H$, define
    \begin{align}
        \operatorname{dis}\left(\varphi\right)=\max_{x\rightarrow x^\prime \in G \atop \varphi\left(x\right)\rightarrow \varphi\left(x^\prime\right)\in H}|w_H\left(\varphi\left(x\right),\varphi\left(x^\prime\right)\right)-w_G\left(x,x^\prime\right)|
    \end{align}
    where $\max_\emptyset=0$.
\end{definition}
\begin{definition}
    For weighted directed graphs $G$ and $H$, and directed graph mappings $\varphi,\psi: G\rightarrow H$, define
    \begin{align}
        \operatorname{cod}\left(\varphi,\psi\right)=\max\left\{\max_{\varphi\left(x\right)\rightarrow\psi\left(x\right) \in H}|{w}_H\left(\varphi\left(x\right),\psi\left(x\right)\right)|,\max_{\psi\left(x\right)\rightarrow\varphi\left(x\right) \in H}|{w}_H\left(\psi\left(x\right),\varphi\left(x\right)\right)|\right\}
    \end{align}
    where $\max_\emptyset=0$.
\end{definition}
\begin{theorem}
    For weighted directed graphs
    \begin{align}
        G=\left(V_G,E_G, w_{G}\right),H=\left(V_H,E_H, w_{H}\right)
    \end{align}
    with the filtration mapping taken as edge sublevel filtration $F_e$. If $G,H$ are homotopy equivalent as directed graphs, i.e., there exist directed graph mappings
    \begin{align}
        \varphi\in \operatorname{Mor}_{ \operatorname{Dgr}}\left(G,H\right),\psi\in \operatorname{Mor}_{ \operatorname{Dgr}}\left(H,G\right)
    \end{align}
    satisfying
    \begin{align}
        \psi\circ\varphi=f_0\simeq_1\ldots\simeq_1f_m=\operatorname{id}_G, \varphi\circ\psi=g_0\simeq_1\ldots\simeq_1g_m=\operatorname{id}_H
    \end{align}
    then:
    \begin{align}
        d_{\mathrm{B}}\left(D_{p}\left(G\right), D_{p}\left(H\right)\right) \leq
         & \max \bigg\{
        \operatorname{dis}\left(\varphi\right),\operatorname{dis}\left(\psi\right),\notag                                                                               \\
         & \frac{1}{2}\max_{k=1,\ldots, m-1}\operatorname{dis}\left(f_k\right),\frac{1}{2}\max_{k=1,\ldots, m-1}\operatorname{dis}\left(g_k\right),\notag               \\
         & \frac{1}{2}\max_{k=1,\ldots, m} \operatorname{cod}\left(f_{k-1},f_k\right),\frac{1}{2}\max_{k=1,\ldots, m}\operatorname{cod}\left(g_{k-1},g_k\right) \bigg\}
    \end{align}
    holds for any $p \in \mathbb{N}$.
\end{theorem}

\begin{proof}
    Let $\eta$ denote the right-hand side. We first consider the case where $\delta>0$.

    If $x\rightarrow x^{\prime} \in G^{\delta}$, we have $w_{G^\delta}\left(x, x^{\prime}\right) \leq \delta$.

    If $\varphi\left(x\right)\rightarrow \varphi\left(x^{\prime}\right)$, by the definition of $\eta$, we have
    \begin{align}
        w_{H}\left(\varphi\left(x\right), \varphi\left(x^{\prime}\right)\right)
        \leq & w_{G}\left(x, x^{\prime}\right)+|w_{G}\left(x, x^{\prime}\right)-w_{H}\left(\varphi\left(x\right), \varphi\left(x^{\prime}\right)\right)|  \notag \\
        \leq & \delta+\operatorname{dis}\left(\varphi\right) \notag                                                                                              \\
        \leq & \delta+\eta
    \end{align}

    Therefore, $\varphi\left(x\right)\rightarrow \varphi\left(x^{\prime}\right)\in H^{\delta+\eta}$, hence $\varphi^{\delta}:x \mapsto \varphi\left(x\right)\in \operatorname{Mor}_{ \operatorname{Dgr}} \left(G^{\delta}, H^{\delta+\eta}\right) $.

    Similarly, $\psi^{\delta}:y \mapsto \psi\left(y\right) \in \operatorname{Mor}_{ \operatorname{Dgr}} \left(H^{\delta}, G^{\delta+\eta}\right)$.

    For $\delta \leq \delta^{\prime} \in \mathbb{R}$, let $\iota_G^{\delta, \delta^{\prime}}:G^{\delta} \hookrightarrow G^{\delta^{\prime}}, \iota_H^{\delta, \delta^{\prime}}:H^{\delta} \hookrightarrow H^{\delta^{\prime}}$ denote the inclusion maps of directed graphs.

    For $x \in V_G^\delta$, we have $\varphi^{\delta^{\prime}}\left(\iota_G^{\delta, \delta^{\prime}}\left(x\right)\right)=\varphi^{\delta^{\prime}}\left(x\right)=\varphi\left(x\right)$.

    Where the second equality follows from the definition of $\varphi^{\delta^{\prime}}$, and the first equality is because $\iota_G^{\delta, \delta^{\prime}}$ is an inclusion map.

    Similarly, for $x \in V_H^\delta$, we have $\iota_H^{\delta+\eta, \delta^{\prime}+\eta}\left(\varphi^{\delta}\left(x\right)\right)=\iota_H^{\delta+\eta, \delta^{\prime}+\eta}\left(\varphi\left(x\right)\right)=\psi\left(x\right)$.

    Therefore, we have $\varphi^{\delta^{\prime}}\left(\iota_G^{\delta, \delta^{\prime}}\left(x\right)\right) =\iota_H^{\delta+\eta, \delta^{\prime}+\eta}\left(\varphi^{\delta}\left(x\right)\right),\forall x \in V_H^\delta$, hence $\varphi^{\delta^{\prime}} \circ \iota_G^{\delta, \delta^{\prime}}\simeq \iota_H^{\delta+\eta, \delta^{\prime}+\eta} \circ \varphi^{\delta}$.

    Similarly, $\psi^{\delta^{\prime}} \circ \iota_H^{\delta, \delta^{\prime}}\simeq \iota_G^{\delta+\eta, \delta^{\prime}+\eta} \circ \psi^{\delta}$.

    Considering the map induced by $f_k$ on directed graphs, denoted $f_k^\delta:G^\delta\rightarrow G^{\delta+2\eta}$.

    If $x\rightarrow x^{\prime} \in G^{\delta}$, we have $w_{G}\left(x, x^{\prime}\right) \leq \delta$.

    If $ f_k\left(x\right)\rightarrow f_k\left(x^{\prime}\right)$, by the definition of $\eta$, we have
    \begin{align}
        w_{G}\left(f_k\left(x\right), f_k\left(x^{\prime}\right)\right)
        \leq & w_{G}\left(x, x^{\prime}\right)+|w_{G}\left(x, x^{\prime}\right)-w_{G}\left(f_k\left(x\right), f_k\left(x^{\prime}\right)\right)| \notag \\
        \leq & \delta+\operatorname{dis}\left(f_k\right) \notag                                                                                         \\
        \leq & \delta+2\eta
    \end{align}

    Therefore, $f_k\left(x\right)\rightarrow f_k\left(x^{\prime}\right)\in G^{\delta+2\eta}$, hence $f_k^{\delta}:x \mapsto f_k\left(x\right)\in \operatorname{Mor}_{ \operatorname{Dgr}} \left(G^{\delta}, G^{\delta+2\eta}\right)$.

    Similarly, $g_k^{\delta}:y \mapsto g_k\left(y\right) \in \operatorname{Mor}_{ \operatorname{Dgr}} \left(H^{\delta}, H^{\delta+2\eta}\right)$.

    For $ x \in V_G^\delta$, since $f_{k-1}\simeq_1 f_k $, we have $f_{k-1}\left(x\right) \stackrel{\smash{\scriptscriptstyle\rightarrow}}{=} f_k\left(x\right)$ or $f_{k}\left(x\right) \stackrel{\smash{\scriptscriptstyle\rightarrow}}{=} f_{k-1}\left(x\right)$.

    If $f_{k-1}\left(x\right)\rightarrow f_{k}\left(x\right)$,
    due to
    \begin{align}
        w_{G}\left(f_{k-1}\left(x\right),f_{k}\left(x\right)\right)
        \leq \operatorname{cod} \left(f_{k-1},f_{k}\right)
        \leq 2\eta
        \leq\delta+2\eta
    \end{align}
    Therefore, $f_{k-1}^\delta\left(x\right) \stackrel{\smash{\scriptscriptstyle\rightarrow}}{=} f_{k}^\delta\left(x\right),\forall x\in V_G^\delta$, thus $\iota_G^{\delta, \delta+2 \eta}\simeq \psi^{\delta+\eta} \circ \varphi^{\delta}$.

    For the case of $f_{k}\left(x\right)\rightarrow f_{k-1}\left(x\right)$, similar discussion yields the same conclusion.

    Similarly, $\iota_H^{\delta, \delta+2 \eta}\simeq \varphi^{\delta+\eta} \circ \psi^{\delta}$.

    Considering the induced maps on the homology groups, by applying \ref{homotopy1}, for any $p \in \mathbb{N}$, we have:
    \begin{align}
        \left(\left(\iota_G^{\delta, \delta+2 \eta}\right)_{ * }\right)_{p}
         & =\left(\left(\psi^{\delta+\eta} \circ \varphi^{\delta}\right)_{ * }\right)_{p}                          \\
        \left(\left(\varphi^{\delta^{\prime}} \circ \iota_G^{\delta, \delta^{\prime}}\right)_{ * }\right)_{p}
         & =\left(\left(\iota_H^{\delta+\eta, \delta^{\prime}+\eta} \circ \varphi^{\delta}\right)_{ * }\right)_{p}
    \end{align}

    Furthermore, by the functoriality of $\left[\mathbf{R}, H_{k}\right]\circ\left[\mathbf{R}, \Omega_*\right]\circ \left[\mathbf{R}, P\right]$, we have:
    \begin{align}
        \left(\left(\iota_G^{\delta, \delta+2 \eta}\right)_{ * }\right)_{p}
         & =\left(\left(\psi^{\delta+\eta}\right)_{ * }\right)_{p} \circ\left(\left(\varphi^{\delta}\right)_{ * }\right)_{p}                          \\
        \left(\left(\varphi^{\delta^{\prime}}\right)_{ * }\right)_{p} \circ\left(\left(\iota_G^{\delta, \delta^{\prime}}\right)_{ * }\right)_{p}
         & =\left(\left(\iota_H^{\delta+\eta, \delta^{\prime}+\eta}\right)_{ * }\right)_{p} \circ\left(\left(\varphi^{\delta}\right)_{ * }\right)_{p}
    \end{align}

    Similarly, for any $p \in \mathbb{N}$,
    \begin{align}
        \left(\left(\iota_H^{\delta, \delta+2 \eta}\right)_{ * }\right)_{p}
         & =\left(\left(\varphi^{\delta+\eta}\right)_{ * }\right)_{p} \circ\left(\left(\psi^{\delta}\right)_{ * }\right)_{p}                       \\
        \left(\left(\psi^{\delta^{\prime}}\right)_{ * }\right)_{p} \circ\left(\left(\iota_H^{\delta, \delta^{\prime}}\right)_{ * }\right)_{p}
         & =\left(\left(\iota_G^{\delta+\eta, \delta^{\prime}+\eta}\right)_{ * }\right)_{p} \circ\left(\left(\psi^{\delta}\right)_{ * }\right)_{p}
    \end{align}

    For $\delta\le0$, we have $E^\delta=\emptyset$, hence $H_p\left(G^\delta\right)=H_p\left(H^\delta\right)=0,\forall p\ge 0$.

    We define
    \begin{align}
        u_\delta=\begin{cases}
                     \left(\left(\varphi^{\delta}\right)_{ * }\right)_{p} & \delta> 0   \\
                     0                                                    & \delta\le 0
                 \end{cases}
    \end{align}
    \begin{align}
        v_\delta=\begin{cases}
                     \left(\left(\psi^{\delta}\right)_{ * }\right)_{p} & \delta> 0   \\
                     0                                                 & \delta\le 0
                 \end{cases}
    \end{align}

    Then
    \begin{align}
        u_{\delta}: & H_{p}\left(G^{\delta}\right) \rightarrow H_{p}\left(H^{\delta+\eta}\right) ,\delta \in \mathbb{R} \\
        v_{\delta}: & H_{p}\left(H^{\delta}\right) \rightarrow H_{p}\left(G^{\delta+\eta}\right) ,\delta \in \mathbb{R}
    \end{align}
    make $\left\{H_{p}\left(G^{\delta}\right)\right\}_{\delta\in \mathbb{R}}$ and $\left\{H_{p}\left(H^{\delta}\right)\right\}_{\delta\in \mathbb{R}}$ $\eta$-interleaved. By \ref{stab}, the theorem is proved.
\end{proof}

In particular, for the case where the underlying directed graphs are the same but with different weights, considering $\varphi=\psi=\operatorname{id}_V$, we immediately obtain:
\begin{corollary}
    For weighted directed graphs
    \begin{align}
        G=\left(V,E, w_{G}\right),G^\prime=\left(V,E, w_{G^\prime}\right)
    \end{align}
    with the filtration mapping taken as the edge-wise horizontal filtration $F_e$, we have:
    \begin{align}
        d_{\mathrm{B}}\left(D_{p}\left(G\right), D_{p}\left(G^\prime \right)\right)
        \leq \max_{x\rightarrow x^\prime \in {G}}|w_{G}\left(x,x^\prime\right)-w_{G^\prime}\left(x,x^\prime \right)|
    \end{align}
    holds for any $p \in \mathbb{N}$. We define $\max_\emptyset=0$.
\end{corollary}

For the case of $G,H$ being complete directed graphs, with any vertex mapping $ \varphi:V_G\to V_H,\psi:V_H\to V_G$ inducing directed graph mappings satisfying $ \psi\circ\varphi\simeq_1\operatorname{id}_G, \varphi\circ\psi\simeq_1\operatorname{id}_H$, we have the following corollary.
\begin{corollary}
    For weighted complete directed graphs $G,H$, with the filtration mapping taken as the edge-wise horizontal filtration $F_e$, we have:
    \begin{align}
        d_{\mathrm{B}}\left(D_{p}\left(G\right), D_{p}\left(H\right)\right) \leq
         & \max_{\varphi:V_G\to V_H \atop \psi:V_H\to V_G } \bigg\{
        \operatorname{dis}\left(\varphi\right),\operatorname{dis}\left(\psi\right),\notag                                                                                        \\
         & \frac{1}{2}\operatorname{cod}\left(\psi\circ\varphi,\operatorname{id}_G\right),\frac{1}{2}\operatorname{cod}\left(\varphi\circ\psi,\operatorname{id}_H\right) \bigg\}
    \end{align}
    holds for any $p \in \mathbb{N}$.
\end{corollary}

\subsection{Stability Results Based on Edge-Weighted Path Complexes}

\begin{definition}
    For edge-weighted path complexes $P,S$ and a weak homomorphism $\varphi: P\rightarrow S $, for $n\in \mathbb{Z}_{>0} $, define
    \begin{align}
        \operatorname{dis}_n\left(\varphi\right)=\max_{e\in P_{n} \atop \varphi\left(e\right)\text{ regular}}|\operatorname{len}\left(\varphi\left(e\right)\right)-\operatorname{len}\left(e\right)|
    \end{align}
    Define $\operatorname{dis}_0\left(\varphi\right)=0$. Define $\max_\emptyset=0$.
\end{definition}
\begin{definition}
    For edge-weighted path complexes $P,S$ and weak homomorphisms $\varphi,\psi: P\rightarrow S $, define
    \begin{align}
        \operatorname{cod}\left(\varphi,\psi\right)=\max\left\{\max_{\varphi\left(x\right)\rightarrow\psi\left(x\right) \in S}|{w}_S\left(\varphi\left(x\right),\psi\left(x\right)\right)|,\max_{\psi\left(x\right)\rightarrow\varphi\left(x\right) \in S}|{w}_S\left(\psi\left(x\right),\varphi\left(x\right)\right)|\right\}
    \end{align}
    Define $\max_\emptyset=0$.
\end{definition}
\begin{theorem}
    For edge-weighted path complexes
    \begin{align}
        P=\left(V_P,\left\{P_{n}\right\}_{n=0}^\infty, w_{P}:P_1\rightarrow \mathbb{R}_{>0}\right),S=\left(V_S,\left\{S_{n}\right\}_{n=0}^\infty, w_{S}:S_1\rightarrow \mathbb{R}_{>0}\right)
    \end{align}
    with filtration map given by pathwise sublevel filtrations $F_p$. For any $p \in \mathbb{N} $, if the subpath complexes of $P,S$
    \begin{align}
        \overline{P}_p =\Delta \left( P_{p}\cup P_{p+1}\right),\overline{S}_p =\Delta \left( S_{p}\cup S_{p+1}\right)
    \end{align}
    are weakly homotopy equivalent, i.e., there exist weak homomorphisms
    \begin{align}
        \varphi\in \operatorname{Mor}_{ \operatorname{\mathcal{WP}}}\left(\overline{P}_p,\overline{S}_p\right),\psi\in \operatorname{Mor}_{ \operatorname{\mathcal{WP}}}\left(\overline{S}_p,\overline{P}_p\right)
    \end{align}
    such that
    \begin{align}
        \psi\circ\varphi=f_0\sim_1\ldots\sim_1f_m=\operatorname{id}_{\overline{P}_p}, \varphi\circ\psi=g_0\sim_1\ldots\sim_1g_m=\operatorname{id}_{\overline{S}_p}
    \end{align}
    then we have:
    \begin{align}
        d_{\mathrm{B}}\left(D_{p}\left(P\right), D_{p}\left(S\right)\right)\leq \max\bigg\{ \max_{i=1,\ldots,p+1}\operatorname{dis}_{i}\left(\varphi\right),\max_{i=1,\ldots,p+1}\operatorname{dis}_{i}\left(\psi\right),\notag \\
        \frac{1}{2}\max_{k=1,\ldots, m}{\max_{l=0,\ldots p}[\operatorname{dis}_{l}\left(f_{k-1}\right)+\operatorname{dis}_{p-l}\left(f_{k}\right)]+\operatorname{cod}\left(f_{k-1},f_k\right)},\notag                           \\
        \frac{1}{2}\max_{k=1,\ldots, m}{\max_{l=0,\ldots p}[\operatorname{dis}_{l}\left(g_{k-1}\right)+\operatorname{dis}_{p-l}\left(g_{k}\right)]+\operatorname{cod}\left(g_{k-1},g_k\right)}\bigg\}
    \end{align}
\end{theorem}

\begin{proof}
    By \ref{determine}, we know that $H_p$ is completely determined by $\mathcal{A}_p, \mathcal{A}_{p+1}$, so $D_p(P)=D_p(\overline{P}_p), D_p(S)=D_p(\overline{S}_p)$, and we only need to consider $\overline{P}_p, \overline{S}_p$.

    Let $\eta$ denote the right-hand side term. We first consider the case $\delta > 0$.

    For $e \in \left(\overline{P}_p^\delta\right)_n$, $n \le p+1$, we have $\operatorname{len}\left(e\right) \leq \delta$.

    If $\varphi\left(e\right)$ is regular, by the definition of $\eta$, we have
    \begin{align}
        \operatorname{len}\left(\varphi\left(e\right)\right)
        \leq & \operatorname{len}\left(e\right)+|\operatorname{len}\left(\varphi\left(e\right)\right)-\operatorname{len}\left(e\right)|  \notag \\
        \leq & \delta+\operatorname{dis}_n\left(\varphi\right) \notag                                                                           \\
        \leq & \delta+\max_{i=1,\ldots,p+1}\operatorname{dis}_{i}\left(\varphi\right) \notag                                                    \\
        \leq & \delta+\eta
    \end{align}

    Therefore, $\varphi\left(e\right)\in \overline{P}_p^{\delta+\eta}$, thus $\varphi^{\delta}:x \mapsto \varphi\left(x\right)\in \operatorname{Mor}_{ \mathcal{WP}} \left(\overline{P}_p^{\delta}, \overline{S}_p^{\delta+\eta}\right)$.

    Similarly, $\psi^{\delta}:y \mapsto \psi\left(y\right) \in \operatorname{Mor}_{ \mathcal{WP}} \left(\overline{S}_p^{\delta}, \overline{P}_p^{\delta+\eta}\right)$.

    For $\delta \leq \delta^{\prime} \in \mathbb{R}$, denote weak morphisms in simplicial complexes as follows:
    \begin{align}
        \iota_{\overline{P}_p}^{\delta, \delta^{\prime}}:\overline{P}_p^{\delta} \hookrightarrow \overline{P}_p^{\delta^{\prime}}, \iota_{\overline{S}_p}^{\delta, \delta^{\prime}}:\overline{S}_p^{\delta} \hookrightarrow \overline{S}_p^{\delta^{\prime}}
    \end{align}

    For all $x \in V_{\overline{P}_p^\delta}$, $\varphi^{\delta^{\prime}}\left(\iota_{\overline{P}_p}^{\delta, \delta^{\prime}}\left(x\right)\right)=\varphi^{\delta^{\prime}}\left(x\right)=\varphi\left(x\right)$.

    The second equality is from the definition of $\varphi^{\delta^{\prime}}$, and the first one is because $\iota_{\overline{P}_p}^{\delta, \delta^{\prime}}$ is an inclusion weak morphism.

    Similarly, $\iota_{\overline{S}_p}^{\delta+\eta, \delta^{\prime}+\eta}\left(\psi^{\delta}\left(x\right)\right)=\iota_{\overline{S}_p}^{\delta+\eta, \delta^{\prime}+\eta}\left(\psi\left(x\right)\right)=\varphi\left(x\right)$.

    Thus, we have $\varphi^{\delta^{\prime}}\left(\iota_{\overline{P}_p}^{\delta, \delta^{\prime}}\left(x\right)\right)=\iota_{\overline{S}_p}^{\delta+\eta, \delta^{\prime}+\eta}\left(\varphi^{\delta}\left(x\right)\right)$.

    Therefore, $\varphi^{\delta^{\prime}} \circ \iota_{\overline{P}_p}^{\delta, \delta^{\prime}}\sim \iota_{\overline{S}_p}^{\delta+\eta, \delta^{\prime}+\eta} \circ \varphi^{\delta}$.

    Similarly, $\psi^{\delta^{\prime}} \circ \iota_{\overline{S}_p}^{\delta, \delta^{\prime}}\sim \iota_{\overline{P}_p}^{\delta+\eta, \delta^{\prime}+\eta} \circ \psi^{\delta}$.

    Considering the weak morphisms induced by $f_k$, $f_k^\delta:{\overline{P}_p}^\delta\rightarrow \overline{P}_p^{\delta+2\eta}$.

    For $e \in \left(\overline{P}_p^{\delta}\right)_n$, $n\le p+1$, we have $\operatorname{len}\left(e\right) \leq \delta$.

    If $f_k\left(e\right)$ is regular, by the definition of $\eta$, we have $\operatorname{len}\left(f_k\left(e\right)\right) \leq \operatorname{len}\left(e\right) +|\operatorname{len}\left(f_k\left(e\right)\right) -\operatorname{len}\left(e\right) |\leq \delta+\operatorname{dis}_n\left(f_k\right)\leq \delta+\operatorname{dis}_{p+1}\left(f_k\right) \leq \delta+2\eta$.

    Therefore, $f_k\left(e\right)\in \overline{P}_p^{\delta+2\eta}$, thus $f_k^{\delta}:x \mapsto f_k\left(x\right)\in \operatorname{Mor}_{ \mathcal{WP}} \left(\overline{P}_p^{\delta}, \overline{P}_p^{\delta+2\eta}\right)$.

    Similarly, $g_k^{\delta}:y \mapsto g_k\left(y\right) \in \operatorname{Mor}_{ \mathcal{WP}} \left(\overline{S}_p^{\delta}, \overline{S}_p^{\delta+2\eta}\right)$.

    By $f_{k-1}\sim_1 f_{k},k=1,\ldots, m$, suppose there exists a weak morphism of simplicial complexes $F_k:\overline{P}_p\times I \rightarrow \overline{S}_p$ satisfying $\left.F_{k}\right|_{\overline{P}_p}=f_{k-1},\left.F_{k}\right|_{\overline{P}_p^{\prime}}=f_{k}^{\prime}$.

    By definition,
    \begin{align}
        \overline{P}_p \times I
        = & \left\{ w \mid w \in \overline{P}_p \right\}
        \cup\left\{ w^{\prime} \mid w^{\prime} \in \overline{P}_p^{\prime}\right\} \notag                                                              \\
          & \cup\left\{i_{0} \ldots i_{l} i_{l}^{\prime} \ldots i_{n}^{\prime} \mid i_{0} \ldots i_{l} i_{l+1} \ldots i_{n} \in \overline{P}_p\right\}
    \end{align}

    We only need to consider the definition of $F_k$ on $\left\{i_{0} \ldots i_{l} i_{l}^{\prime} \ldots i_{n}^{\prime} \mid i_{0} \ldots i_{l} i_{l+1} \ldots i_{n} \in \overline{P}_p\right\}$.

    Let $F_k: i_{0} \ldots i_{l} i_{l}^{\prime} \ldots i_{n}^{\prime} \mapsto f_{k-1}\left(i_{0}\right) \ldots f_{k-1}\left(i_{l}\right) f_{k}\left(i_{l}\right) \ldots f_{k}\left(i_{n}\right)$.

    For $i_{0} \ldots i_{l} i_{l+1} \ldots i_{n} \in \left(\overline{P}_p^{\delta}\right)_n$, $n\le p+1$, we have $\operatorname{len}\left(i_{0} \ldots i_{l} i_{l+1} \ldots i_{n}\right) \leq \delta$.

    If $F_k\left(i_{0} \ldots i_{l} i_{l}^{\prime} \ldots i_{n}^{\prime}\right)=f_{k-1}\left(i_{0}\right) \ldots f_{k-1}\left(i_{l}\right) f_{k}\left(i_{l}\right) \ldots f_{k}\left(i_{n}\right)$ is regular, we have
    \begin{align}
            & \operatorname{len}\left(f_{k-1}\left(i_{0}\right) \ldots f_{k-1}\left(i_{l}\right) f_{k}\left(i_{l}\right) \ldots f_{k}\left(i_{n}\right)\right)\notag                                                                             \\
        \le & \operatorname{len}\left(i_{0} \ldots i_{l} i_{l+1} \ldots i_{n}\right)\notag                                                                                                                                                       \\
            & +|\operatorname{len}\left(f_{k-1}\left(i_{0}\right) \ldots f_{k-1}\left(i_{l}\right) f_{k}\left(i_{l}\right) \ldots f_{k}\left(i_{n}\right)\right) -\operatorname{len}\left(i_{0} \ldots i_{l} i_{l+1} \ldots i_{n}\right)| \notag \\
        \le & \delta+ \operatorname{dis}_l\left(f_{k-1}\right)+\operatorname{dis}_{n-l}\left(f_{k}\right)+|w_P\left(f_{k-1}\left(i_l\right),f_{k}\left(i_l\right)\right)|  \notag                                                                \\
        \le & \delta+ \operatorname{dis}_l\left(f_{k-1}\right)+\operatorname{dis}_{n-l}\left(f_{k}\right)+\operatorname{cod}\left(f_{k-1},f{k}\right) \notag                                                                                     \\
        \le & \delta+2\eta
    \end{align}

    Therefore, $f_{k-1}^{\delta}\sim_1 f_{k}^{\delta}$.

    For $F_k:\overline{P}_p\times I \rightarrow \overline{S}_p$ satisfying $\left.F_{k}\right|_{\overline{P}_p}=f_{k},\left.F_{k}\right|_{\overline{P}_p^{\prime}}=f_{k-1}^{\prime}$, the discussion is similar, and we obtain the same conclusion.

    Similarly, we have $g_{k-1}^{\delta}\sim_1 g_{k}^{\delta}$.

    Thus, $\iota_{\overline{P}_p}^{\delta, \delta+2 \eta}\sim \psi^{\delta+\eta} \circ \varphi^{\delta}$,
    $\iota_{\overline{S}_p}^{\delta, \delta+2 \eta}\sim \varphi^{\delta+\eta} \circ \psi^{\delta}$.

    Applying \ref{homotopy2}, we have the following results in the homotopy groups:
    \begin{align}
        \left(\left(\iota_{\overline{P}_p}^{\delta, \delta+2 \eta}\right)_{ * }\right)_{p}
         & =\left(\left(\psi^{\delta+\eta} \circ \varphi^{\delta}\right)_{ * }\right)_{p}                                         \\
        \left(\left(\varphi^{\delta^{\prime}} \circ \iota_{\overline{P}_p}^{\delta, \delta^{\prime}}\right)_{ * }\right)_{p}
         & =\left(\left(\iota_{\overline{S}_p}^{\delta+\eta, \delta^{\prime}+\eta} \circ \varphi^{\delta}\right)_{ * }\right)_{p}
    \end{align}

    Also, by functoriality of $\left[\mathbf{R}, H_{k}\right]\circ\left[\mathbf{R}, \Omega_*\right]$, we have:
    \begin{align}
        \left(\left(\iota_{\overline{P}_p}^{\delta, \delta+2 \eta}\right)_{ * }\right)_{p}
         & =\left(\left(\psi^{\delta+\eta}\right)_{ * }\right)_{p} \circ\left(\left(\varphi^{\delta}\right)_{ * }\right)_{p}                                         \\
        \left(\left(\varphi^{\delta^{\prime}}\right)_{ * }\right)_{p} \circ\left(\left(\iota_{\overline{P}_p}^{\delta, \delta^{\prime}}\right)_{ * }\right)_{p}
         & =\left(\left(\iota_{\overline{S}_p}^{\delta+\eta, \delta^{\prime}+\eta}\right)_{ * }\right)_{p} \circ\left(\left(\varphi^{\delta}\right)_{ * }\right)_{p}
    \end{align}

    Similarly, we have,
    \begin{align}
        \left(\left(\iota_{\overline{S}_p}^{\delta, \delta+2 \eta}\right)_{ * }\right)_{p}
         & =\left(\left(\varphi^{\delta+\eta}\right)_{ * }\right)_{p} \circ\left(\left(\psi^{\delta}\right)_{ * }\right)_{p}                                      \\
        \left(\left(\psi^{\delta^{\prime}}\right)_{ * }\right)_{p} \circ\left(\left(\iota_{\overline{S}_p}^{\delta, \delta^{\prime}}\right)_{ * }\right)_{p}
         & =\left(\left(\iota_{\overline{P}_p}^{\delta+\eta, \delta^{\prime}+\eta}\right)_{ * }\right)_{p} \circ\left(\left(\psi^{\delta}\right)_{ * }\right)_{p}
    \end{align}

    For the case $\delta\le0$, we have $\overline{P}_p^\delta=V_{\overline{P}_p},\overline{S}_p^\delta=V_{\overline{S}_p}$, and thus $H_p\left(\overline{P}_p^\delta\right)=H_p\left(\overline{S}_p^\delta\right)=0,\forall p\ge 0$.

    We define
    \begin{align}
        u_\delta=\begin{cases}
                     \left(\left(\varphi^{\delta}\right)_{ * }\right)_{p} & \delta> 0   \\
                     0                                                    & \delta\le 0
                 \end{cases}
    \end{align}
    \begin{align}
        v_\delta=\begin{cases}
                     \left(\left(\psi^{\delta}\right)_{ * }\right)_{p} & \delta> 0   \\
                     0                                                 & \delta\le 0
                 \end{cases}
    \end{align}

    Then
    \begin{align}
        u_{\delta}: & H_{p}\left(\overline{P}_p^\delta\right) \rightarrow H_{p}\left(\overline{S}_p^{\delta+\eta}\right) ,\delta \in \mathbb{R} \\
        v_{\delta}: & H_{p}\left(\overline{S}_p^\delta\right) \rightarrow H_{p}\left(\overline{P}_p^{\delta+\eta}\right) ,\delta \in \mathbb{R}
    \end{align}
    make $\left\{H_{p}\left(\overline{P}_p^\delta\right)\right\}_{\delta\in \mathbb{R}}$ and $\left\{H_{p}\left(\overline{S}_p^\delta\right)\right\}_{\delta\in \mathbb{R}}$ $\eta$-interleaved.

    Thus, by \ref{stab}, the theorem is proved.
\end{proof}

In particular, for the case where the underlying road complex is the same but has different weights, considering $\varphi=\psi=\operatorname{id}_V$, we immediately get:
\begin{corollary}
    For chain complexes
    \begin{align}
        P=\left(V,\left\{P_{n}\right\}_{n=0}^\infty, w_{P}:P_1\rightarrow \mathbb{R}_{>0}\right),P^\prime=\left(V,\left\{P_{n}\right\}_{n=0}^\infty, w_{P^\prime}:P_1\rightarrow \mathbb{R}_{>0}\right)
    \end{align}
    with the same underlying road complex but different weights, considering $\varphi=\psi=\operatorname{id}_V$, we immediately obtain:
    \begin{align}
        d_{\mathrm{B}}\left(D_{p}\left(P\right), D_{p}\left(P^\prime\right)\right)
        \leq \max_{e\in P }|\operatorname{len}_{P}\left(e\right)-\operatorname{len}_{P^\prime}\left(e\right)|
    \end{align}
    holds for any $p \in \mathbb{N}$.
\end{corollary}

For the case of $P,S$ being complete chain complexes, any vertex maps $ \varphi:V_P\to V_S,\psi:V_S\to V_P$ induce weak morphisms of chain complexes and satisfy $ \psi\circ\varphi\sim_1\operatorname{id}_P, \varphi\circ\psi\sim_1\operatorname{id}_S$. Hence, we have the following corollary.
\begin{corollary}
    For edge-weighted complete chain complexes $P,S$, with the filtration map taken as the down-horizon filtration $F_p$, we have:
    \begin{align}
         & d_{\mathrm{B}}\left(D_{p}\left(P\right), D_{p}\left(S\right)\right)\leq \max_{\varphi:V_P\to V_S\atop \psi:V_S\to V_P}\bigg\{ \max_{i=1,\ldots,p+1}\operatorname{dis}_{i}\left(\varphi\right),\max_{i=1,\ldots,p+1}\operatorname{dis}_{i}\left(\psi\right),\notag \\
         & \frac{1}{2}{\max_{l=0,\ldots p}[\operatorname{dis}_{l}\left(\psi\circ\varphi\right)+\operatorname{dis}_{p-l}\left(\operatorname{id}_{P}\right)]+\operatorname{cod}\left(f_{k-1},f_k\right)},\notag                                                                \\
         & \frac{1}{2}{\max_{l=0,\ldots p}[\operatorname{dis}_{l}\left(\varphi\circ\psi\right)+\operatorname{dis}_{p-l}\left(\operatorname{id}_{S}\right)]+\operatorname{cod}\left(g_{k-1},g_k\right)}\bigg\}
    \end{align}
    holds for any $p \in \mathbb{N}$.
\end{corollary}

\bibliography{stabPPH}

\end{document}